\documentclass[10pt]{amsart}
\usepackage{amssymb,amsmath,textcomp,amsfonts,amsthm}
\usepackage{amssymb,amsmath,amsopn,textcomp, thmtools}
\usepackage[utf8]{inputenc}
\usepackage{url}
\usepackage{mathtools}
\usepackage{romannum}
\usepackage{color}
\usepackage[normalem]{ulem}
\usepackage[foot]{amsaddr}
\newtheoremstyle{normal}
{10pt}
{10pt}
{\normalfont}
{}
{\bfseries}
{}
{0.8em}
{\bfseries{\thmname{#1}\thmnumber{ #2}.\thmnote{ \hspace{0.5em}(#3)\newline}}}
 
\newtheoremstyle{kursiv}
{10pt}
{10pt}
{\itshape}
{}
{\bfseries}
{}
{0.8em}
{\bfseries{\thmname{#1}\thmnumber{ #2}.\thmnote{ \hspace{0.5em}(#3)\newline}}}

\theoremstyle{kursiv}
\newtheorem{definition}{Definition}[section]
\newtheorem{satz}[definition]{Theorem}
\newtheorem{lemma}[definition]{Lemma}
\newtheorem{korollar}[definition]{Corollary}
\newtheorem{proposition}[definition]{Proposition}

\theoremstyle{normal}
\newtheorem{bemerkungen}[definition]{Remarks}
\newtheorem{bemerkung}[definition]{Remark}
\newtheorem{beispiel}[definition]{Example}

\newlength{\leftstackrelawd}
\newlength{\leftstackrelbwd}
\def\leftstackrel#1#2{\settowidth{\leftstackrelawd}%
{${{}^{#1}}$}\settowidth{\leftstackrelbwd}{$#2$}%
\addtolength{\leftstackrelawd}{-\leftstackrelbwd}%
\leavevmode\ifthenelse{\lengthtest{\leftstackrelawd>0pt}}%
{\kern-.5\leftstackrelawd}{}\mathrel{\mathop{#2}\limits^{#1}}}

\newcommand{\iR}{\mathbb{R}}
\newcommand{\R}{\mathbb{R}}
\newcommand{\iN}{\mathbb{N}}
\newcommand{\N}{\mathbb{N}}
\newcommand{\iZ}{\mathbb{Z}}
\newcommand{\Z}{\mathbb{Z}}

\allowdisplaybreaks
\begin{document}
\pagenumbering{arabic}
\title{Li-Yau and Harnack inequalities via curvature-dimension conditions for discrete long-range
jump operators including the fractional discrete Laplacian}
\date{\today}
\author{Sebastian Kr\"ass}
\email{sebastian.kraess@uni-ulm.de}
\author{Frederic Weber}
\email{frederic.weber@alumni.uni-ulm.de}
\author{Rico Zacher$^*$}
\thanks{$^*$Corresponding author}
\email[Corresponding author:]{rico.zacher@uni-ulm.de}
\address[Sebastian Kr\"ass, Frederic Weber, Rico Zacher]{Institute of Applied Analysis, Ulm University, Helmholtzstra\ss{}e 18, 89081 Ulm, Germany.}

\begin{abstract}
We consider operators of the form 
$L u(x) = \sum_{y \in \mathbb{Z}} k(x-y) \big( u(y) - u(x)\big)$ on the one-dimensional lattice with symmetric, integrable kernel $k$. We prove several results stating that under certain conditions on the kernel the operator $L$ satisfies the  curvature-dimension
condition $CD_\Upsilon (0,F)$ (recently introduced in \cite{WEB,WZ}) with some $CD$-function $F$, where attention is also 
paid to the asymptotic properties of $F$ (exponential growth at infinity and power-type behaviour near zero). We show that $CD_\Upsilon (0,F)$ implies a Li-Yau inequality for positive solutions of the heat equation associated with the operator
$L$. The Li-Yau estimate in turn leads to a Harnack inequality, from which we also derive heat kernel bounds. Our results
apply to a wide class of operators including the fractional discrete Laplacian.
\end{abstract}
\maketitle
\noindent{\bf AMS subject classification (2020):} 60J27 (primary), 60J74, 47D07 (secondary)

${}$

\noindent{\bf Keywords:} curvature-dimension inequality, $CD$-function, long-range jumps, fractional discrete Laplacian, 
Li-Yau inequality, Harnack inequality 
\section{Introduction}
\subsection{Curvature-dimension conditions and Li-Yau inequalities}
The $\Gamma$-calculus of Bakry and \'Emery is nowadays known as a powerful theory at the interface of analysis, geometry and probability theory. The central role is played by the curvature-dimension condition $CD(\kappa,d)$, which takes its name from the fact that the Laplace-Beltrami operator $\Delta$ on a complete Riemannian manifold $M$ satisfies $CD(\kappa,d)$ if and only if the Ricci curvature of the manifold is bounded from below by $\kappa$ and $d$ is greater than or equal to the topological dimension of $M$. As a consequence, the Bakry-\'Emery theory includes the celebrated Li-Yau inequality, which says that if $M$ has nonnegative Ricci curvature and topological dimension $d$, then the estimate
\begin{equation}\label{originalLiYau}
- \Delta \log u \leq \frac{d}{2t} , \quad t>0,
\end{equation} 
holds true for any  positive solution $u$ to the heat equation $\partial_t u = \Delta u$ on $(0,\infty)\times M$, where $\Delta$ again denotes the corresponding Laplace-Beltrami operator. This inequality has been established in the seminal paper \cite{LY} and its popularity is due to the fact that (sharp) parabolic Harnack inequalities can be deduced from it, see also \cite{Li}. The generalisation of \eqref{originalLiYau} to Markov diffusion operators satisfying $CD(0,d)$ has been elaborated in \cite{BBL}. This result, among many other impressive applications of curvature-dimension conditions like, e.g., certain functional inequalities can be found in the extensive monograph \cite{BGL}. 

A crucial assumption the Bakry-\'Emery theory relies on is the validity of certain chain rule formulas, in \cite{BGL} subsumed under the notion of the diffusive property. This assumption is the major obstacle if
one wants to extend the Bakry-\'Emery theory to
non-local operators. Nevertheless, inequalities like \eqref{originalLiYau} are of course also of great interest in the non-local context. To emphasize the latter, we point out that the question whether a Li-Yau inequality holds for one of the most well known non-local operators, the fractional Laplacian on $\mathbb{R}^d$, has been answered very recently in \cite{WZ2}. There, it has been observed that positive solutions to the fractional heat equation satisfy a Li-Yau type inequality with the same kind of temporal behaviour on the right-hand side as described in \eqref{originalLiYau} (a constant divided by $t$), although it has been shown in \cite{SWZ2} that  the fractional Laplacian on $\mathbb{R}^d$ fails to satisfy the Bakry-\'Emery curvature-dimension condition $CD(\kappa,N)$ for any $\kappa \in \mathbb{R}$ and $N<\infty$. 

Nonetheless, curvature-dimension conditions in the context of Li-Yau inequalities have attracted a lot of research in the last decade also in non-local settings, like the discrete one, see  \cite{BHL,DKZ,MUN,MUN2}. In all these references 
the original Bakry-\'Emery condition was replaced by alternative conditions in order to overcome the lack of chain rule in the discrete setting. In \cite{BHL}, the authors introduced the exponential curvature-dimension condition, which is based on the observation that a specific chain rule formula for the square root also holds true in the discrete framework of graphs. The Li-Yau inequality of \cite{BHL} states that
\begin{equation*}
\frac{- \Delta \sqrt{u}}{\sqrt{u}} \leq \frac{d}{2t}, \quad t>0,
\end{equation*}
for any positive solution to the heat equation on a locally finite graph on which the exponential curvature-dimension condition is satisfied with dimension term $d$ and nonnegative curvature. Here, $\Delta$ denotes the corresponding graph Laplacian. A bit later M\"unch introduced the $\Gamma^\psi$-calculus in \cite{MUN}, where $\psi\in C^1(\mathbb{R}_+)$ is a concave function. Choosing $\psi$ as the square root function, the calculus of \cite{MUN} reduces to \cite{BHL} and hence the above Li-Yau inequality (in the case of finite graphs) is a special case of the results obtained by M\"unch via his $CD\psi$-condition. For the choice of $\psi=\log$, the resulting Li-Yau inequality in \cite{MUN} takes the same form as the original one in \eqref{originalLiYau} with $\Delta$ being the graph Laplacian. The condition $CD(F;0)$ was introduced in \cite{DKZ} and it has been observed there that this condition leads to Li-Yau inequalities with a more general right-hand side. In fact, one of the main contributions of \cite{DKZ} is that the quadratic dimension part in the curvature-dimension condition (also used in \cite{BHL, MUN}) can be replaced by a more general term involving a $CD$-function $F$, which allows for a more precise analysis with regard to the right-hand side of the Li-Yau inequality. To be more concrete, the Li-Yau estimates of \cite{DKZ} read as
\begin{equation*}
- \Delta \log u \leq \varphi(t),\quad t>0,
\end{equation*}
where $\Delta$ again is the graph Laplacian and $\varphi$ denotes the so-called relaxation function associated with the $CD$-function $F$. We will repeat those definitions in the next subsection as they will also play a central role in the present paper. 

The works \cite{BHL,DKZ,MUN} have in common that Harnack inequalities can be deduced from their Li-Yau estimates in a quite similar fashion. What they also share is that they work on locally finite, if not finite, state spaces. In \cite{SWZ1} the authors have studied the classical condition $CD(0,d)$ for a special case of locally infinite graphs. Although Li-Yau estimates have not been addressed in \cite{SWZ1} this work serves as the main inspiration for the present paper alongside \cite{DKZ,WEB,WZ}.

\subsection{Setting and main results}
We study operators of the form
\begin{equation}\label{eq:operator}
L u(x) = \sum\limits_{y \in \mathbb{Z}} k(x-y) \big( u(y) - u(x)\big), \quad x \in \mathbb{Z},
\end{equation}
with a (non-trivial) kernel $k: \mathbb{Z} \to [0,\infty)$ that we assume throughout this paper to be symmetric and integrable, i.e. we have 
\begin{equation} \label{StandingHyp}
k(j)=k(-j) \text{ for any } j \in \mathbb{N} \text{  and }|k|_1 := \sum\limits_{j \in \mathbb{Z}} k(j)< \infty.
\end{equation} 
It is convenient to set $k(0)=0$. However, we will also interpret \eqref{eq:operator} as the generator of a continuous-time Markov chain with state space $\mathbb{Z}$. In this case, the transition rates $\big( q(x,y)\big)_{x,y \in \mathbb{Z}}$ are given by $q(x,y)=k(x-y)$ if $x \neq y$ and $q(x,x)=-\sum_{j\in \mathbb{Z}\setminus\{0\}}k(j)$ for any $x \in \mathbb{Z}$. Since the value of $k$ at $0$ does not play a role in \eqref{eq:operator} we can choose to set $k(0)=0$ instead of
$k(0)=-\sum_{j\in \mathbb{Z}\setminus\{0\}}k(j)$, the latter being the natural choice in the Markov chain setting.

The generator of a Markov chain naturally induces a graph structure (without loops), where here the set of vertices equals the state space $\mathbb{Z}$ and the edge weights between two distinct vertices are given by the respective transition rates. We call the resulting graph the underlying graph to $L$. Since $k$ is symmetric, the underlying graph has symmetric edge weights. We will say that the operator $L$ is generated by the kernel $k$. 

The simplest example for an operator of the type \eqref{eq:operator} is given by the discrete Laplacian $\Delta$ (from now on $\Delta$ denotes the discrete Laplacian and not the graph Laplacian), which is generated by the kernel $k(j)=1$ if $|j|=1$ and $k(j)=0$ else. The underlying graph to $\Delta$ is given by the unweighted graph with vertex set $\mathbb{Z}$, where each vertex $x \in \mathbb{Z}$ is adjacent only to $x-1$ and $x+1$. 

Our main interest focuses on much more complicated situations. A graph is called locally infinite if each of its vertices is adjacent to infinitely many other vertices. The underlying graph to $L$ is locally infinite if and only if $k$ has unbounded support. Our leading example, the case of a power-type kernel, has this property. For $\beta>0$ we define the power-type kernel as
\begin{equation} \label{powerkernel}
k_\beta(j) = \frac{1}{|j|^{1+\beta}}, \quad j \in \mathbb{Z}\setminus \{0\}.
\end{equation} 
The operator generated by $k_\beta$ will be denoted by $L_\beta$, $\beta>0$. The importance of these specific kernels follows from their relation to fractional powers of the discrete Laplacian. In \cite{CRS}, the operator $-(-\Delta)^\frac{\beta}{2}$, $\beta \in (0,2)$, has been defined by means of the semigroup method and it has been shown there that for each $\beta \in (0,2)$ the operator $-(-\Delta)^\frac{\beta}{2}$ can be written in the form \eqref{eq:operator} with kernel $k_\beta^*$ given by
\begin{equation}\label{FrakLaplacekernel}
k_\beta^*(j)=\frac{4^\frac{\beta}{2}\Gamma\big(\frac{1+\beta}{2}\big)}{\sqrt{\pi}\big| \Gamma\big(-\frac{\beta}{2}\big)\big|} \frac{\Gamma\big(|j|-\frac{\beta}{2}\big)}{\Gamma\big(|j|+1+\frac{\beta}{2}\big)}, \quad j \in \mathbb{Z}\setminus\{0\}.
\end{equation}
Here $\Gamma$ denotes the Gamma function. It is known that there exist constants $C(\beta), c(\beta)>0$ such that
\begin{equation}
c(\beta) k_\beta(j) \leq k_\beta^*(j) \leq C(\beta) k_\beta(j), \ j \in \mathbb{Z},
\label{Zusammenhang Kerne}
\end{equation}
see \cite[Theorem 1.1]{CRS}.

As explained before, the purpose of \cite{SWZ1} was to study the validity of the Bakry-\'Emery condition $CD(0,d)$, which, by definition, is satisfied if
\begin{equation*}
\Gamma_2(u,u)(x) \geq \kappa \Gamma(u,u)(x) + \frac{1}{d} \big( L u(x)\big)^2
\end{equation*}
for any $x \in \mathbb{Z}$ and for any $u$ in a sufficiently rich class of functions, like e.g. the class of bounded functions on $\mathbb{Z}$. Here, the carr\'e du champ operator $\Gamma$ and the iterated carr\'e du champ operator $\Gamma_2$ 
are defined as
\begin{equation*}
\begin{split}
\Gamma(u,v)&= \frac{1}{2}\big( L(uv)-uLv -v L u \big)\\
\Gamma_2(u,v)&= \frac{1}{2}\big( L\Gamma(u,v) - \Gamma(u,Lv) - \Gamma(v,Lu)\big),
\end{split}
\end{equation*}
where $u,v$ are elements of a suitable algebra of real-valued functions. 
The main positive result of \cite{SWZ1} is that whenever the kernel $k$ is non-increasing (on $\iN$) and has finite second moment (i.e. $\sum_{j\in \iZ} k(j) j^2 < \infty$), then $CD(0,d)$ holds true for some $d<\infty$. This particularly applies to the operator $L_\beta$ in the case of $\beta>2$. Contrariwise, it has also been obtained in \cite{SWZ1} that $CD(0,d)$ with finite $d$ is impossible to hold for $L_\beta$ if $\beta \in (0,2)$. This negative result has also been extended to the fractional discrete Laplacian.

As we already have discussed above, the classical Bakry-\'Emery condition does not seem to be suitable for deriving  Li-Yau estimates in the discrete setting. However, it is known that the condition of M\"unch with $\psi=\log$ implies the condition $CD(0,d)$, see \cite{MUN}. This seems quite problematic when one aims to derive Li-Yau inequalities for the fractional discrete Laplacian by means of curvature-dimension techniques. The notion of $CD$-functions developed in \cite{DKZ} is the way out of this misery. Very recently, the curvature-dimension condition $CD_\Upsilon(\kappa,F)$ has been introduced in the works \cite{WZ} and \cite{WEB}, where $\kappa \in \mathbb{R}$ denotes the curvature constant and $F$ a $CD$-function. This condition is in fact a consistent analogue to the classical Bakry-\'Emery condition since it implies modified logarithmic Sobolev inequalities under positive curvature bounds (see \cite{WZ}), logarithmic entropy-information inequalities  under positive curvature and finite dimension bounds (see \cite{WEB}) and Li-Yau inequalities under nonnegative curvature and finite dimension bounds in the locally finite state space case due to their strong connections to \cite{MUN} and \cite{DKZ}, see \cite[Remark 2.9]{WZ}. A key role in the $CD_\Upsilon$ condition is played by the function $\Upsilon(r)=e^r-1-r$, $r \in \mathbb{R}$, which appears in the operator
\begin{equation*}
\Psi_\Upsilon(u)(x) = \sum\limits_{y \in \mathbb{Z}} k(x-y) \Upsilon\big(u(y)-u(x)\big),\quad x\in \iZ.
\end{equation*} 
This operator serves as a replacement for the carr\'e du champ operator $\Gamma(u,u)$, which can be seen in the non-local chain rule type formula
\begin{equation*}
L (\log u) = \frac{L u }{u}- \Psi_\Upsilon (\log u),
\end{equation*} 
see \cite{DKZ} and also \cite[Lemma 2.2]{WZ} for the respective regularity assumptions in the  more general setting of locally infinite graphs.
The operator
\begin{equation*}
\Psi_{2,\Upsilon}(u)(x) = \frac{1}{2}\big( L \Psi_\Upsilon(u)(x) - B_{\Upsilon'}(u,Lu)(x)\big),\quad x\in \iZ,
\end{equation*}
serves as the counterpart for $\Gamma_2(u,u)$, where
\begin{equation*}
B_{\Upsilon'}(u,v)(x)= \sum\limits_{y \in \mathbb{Z}} k(x-y) \Upsilon'(u(y)-u(x)) \big( v(y)-v(x)\big).
\end{equation*}
In the general setting of \cite{WZ} and \cite{WEB} these operators are defined on the space of bounded functions, which will be denoted by $\ell^\infty(\mathbb{Z})$ in our situation. Note that the integrability of the kernel in fact implies that $L u (x), \Psi_\Upsilon(u)(x)$ and $\Psi_{2,\Upsilon}(u)(x)$ are well-defined for any $u \in \ell^\infty(\iZ)$ and $x \in \mathbb{Z}$. However, it is also natural to consider the larger space
\begin{equation*}
\ell_k^1(\mathbb{Z}) = \big\{ u: \mathbb{Z}\to \mathbb{R} \text{ such that } \sum\limits_{j \in \mathbb{Z}} k(j) |u(j)| <\infty \big\},
\end{equation*} 
since  $u(x+\cdot)\in \ell_k^1(\mathbb{Z})$ at $x \in \mathbb{Z}$ implies that $L u (x)$ exists.

For the following definition we will work with the original function space.
We say that $L$ satisfies $CD_\Upsilon(\kappa,F)$ if
\begin{equation*}
\Psi_{2,\Upsilon} (u) (x) \geq \kappa \Psi_\Upsilon (u) (x) + F_0 (- L u(x))
\end{equation*}
 holds at any $x \in \mathbb{Z}$ and for any $u\in \ell^\infty(\mathbb{Z})$, where $F_0=F$ on $[0,\infty)$ and $F_0(r)=0$ if $r<0$. Since we are interested in proving Li-Yau inequalities, we will be working with the condition $CD_\Upsilon(0,F)$. The function $F:[0,\infty) \to [0,\infty)$ is a so called $CD$-function which is defined by the properties that $F$ is continuous,
$F(0)=0$, the mapping $r \mapsto \frac{F(r)}{r}$ is strictly increasing in $(0,\infty)$ and $\frac{1}{F}$ is integrable at $\infty$. For each such function $F$ there exists a uniquely determined mapping $\varphi :(0,\infty) \to (0,\infty)$ that solves the ordinary differential equation $\dot{\varphi}(t)+F(\varphi(t)) = 0$ for any $t>0$ and blows up at $t=0$, see \cite[Lemma 3.5]{DKZ}. We call $\varphi$ the relaxation function to $F$.

It is straightforward to check that the specific structure of the transition rates allows for an alternative representation of \eqref{eq:operator}, which has been already used in \cite{SWZ1} and reads as
\begin{equation*}
L u(x) = \sum\limits_{j \in \mathbb{Z}} k(j) \big( u(x+j)-u(x)\big),\quad x\in \iZ.
\end{equation*}
But also for the more complicated operator $\Psi_{2,\Upsilon}$ there is a quite useful representation formula available that has been derived in \cite[Proposition 2.16]{WZ} and is given by
\begin{equation}
\Psi_{2,\Upsilon}(u) (x) = \frac{1}{2} \sum\limits_{j,l \in \mathbb{Z}} k(j) k(l) e^{u(x+l)-u(x)} \Upsilon\big( u(x+j+l)-u(x+j)-u(x+l)+u(x)\big), \quad x \in \mathbb{Z},
\label{Psi_2H Alternativ}
\end{equation}
where $u$ is a bounded function. We can read from this formula that $\Psi_{2,\Upsilon}(u)$ is nonnegative. Consequently, in order to prove $CD_\Upsilon(0,F)$ for bounded functions $u$ at the point $x\in \iZ$, we only need to consider the case where $-Lu(x) >0$. 

Based on our discussion on the results of \cite{DKZ} and \cite{SWZ1} the following questions are natural:
\begin{itemize}
\item Can we find a $CD$-function $F$ such that the fractional discrete Laplacian satisfies the condition $CD_\Upsilon(0,F)$?
\item Does every operator that is generated by a non-increasing kernel with finite second moment satisfy $CD_\Upsilon(0,F)$ for some $CD$-function $F$ that behaves quadratically near $0$?
\item Does $CD_\Upsilon(0,F)$ also imply Li-Yau inequalities in our specific context of locally infinite underlying graphs?
\end{itemize}
In this article, we will answer all of the above questions in the affirmative. 
 
We now outline the structure and main results of this article. In Section 2 we show some first basic results for kernels with finite support (Corollary \ref{Finitesupportresult}) and also for kernels satisfying a very mild integrability condition (Theorem \ref{basicCDresult}). Both these results are based on a basic lower estimate on $\Psi_{2,\Upsilon}$. In Section 3 we investigate which is the optimal $CD$-function obtained from this basic estimate alone. This is done by using Lagrange multiplier techniques. The resulting $CD$-function is a strictly convex function that behaves exponentially at $\infty$ (Theorem \ref{CDresultLagrange}).  In order to improve the asymptotic behaviour near zero we apply quite intricate calculations in Section 4 to generalize the positive result for non-increasing kernels with finite second moment from \cite{SWZ1} to the condition $CD_\Upsilon(0,F)$ under a mild additional assumption. Most notably, our result also applies to kernels that only have finite $\alpha$-th moment for $\alpha \in [1,2)$ (see Theorem \ref{Satz alphates Moment}). This makes it also applicable to power-type kernels and fractional powers of the discrete Laplacian if $\beta>1$ (Theorem \ref{SatzCDFunctionPotenzkern} and Corollary \ref{KorollarCDFunktionFrakLaplace}). By applying subtle cut-off arguments, we show that $CD_\Upsilon(0,F)$ implies a Li-Yau inequality also in our locally infinite setting (Theorem \ref{SatzLiYau}). From this Li-Yau inequality we further derive a Harnack inequality in Section 6 (Theorem \ref{SatzHarnackUngl}). 
In sharp contrast to the classical case based on the Bakry-\'Emery condition, our Harnack inequality is valid without initial
time gap provided that the $CD$-function $F$ behaves exponentially at infinity (see Remark \ref{BemerkungHarnack}), which is in fact the case for a wide class
of kernels (including those of power-type) as we know from Section 3. We finally show how this stronger version of
Harnack inequality leads to heat kernel lower bounds in a rather direct way.

We note that there are already known results on Harnack inequalities and heat kernel bounds for longe-range jump processes on discrete state spaces, see e.g.\ \cite{BBK,BL1,BL2,CK}. In contrast to these references, which use probabilistic methods,
our proof of the Harnack inequality is purely analytic. We also remark that in the case of locally finite graphs, parabolic type Harnack inequalities 
have already been established in \cite{DEL} using the method of Moser. 

We emphasize that in our work the main focus is on establishing Li-Yau inequalities,
that is, differential Harnack inequalities. Our approach based on suitable curvature-dimension conditions seems to be completely new for long-range
jump operators. In particular, we elaborate on the important role played by the
$CD$-function $F$ occurring in the $CD$-inequality and how the kernel in the jump operator
determines certain properties of $F$, which in turn are crucial with regard to the form
of the resulting Li-Yau and Harnack inequalities.

\subsection*{Acknowledgement} 
Sebastian Kr\"ass is supported by a PhD-scholarship of the ``Hanns-Seidel-Stiftung'', Germany. Fre\-deric Weber was supported by a PhD-scholarship of the ``Studien-stiftung des
deutschen Volkes'', Germany. Rico Zacher is supported by the DFG (project number 355354916, GZ ZA 547/4-2).
\section{Preliminaries and basic results}
We first repeat the definition of a $CD$-function and the corresponding relaxation function introduced in \cite[Definition 3.1 and Definition 3.6]{DKZ}.

\begin{definition}
A continuous function $F:[0,\infty) \to [0,\infty)$ is called $CD$-function, if $F(0)=0$, $\frac{F(x)}{x}$ is strictly increasing on $(0,\infty)$ and $\frac{1}{F}$ is integrable at $\infty$. The unique strictly positive solution $\varphi$ of the ODE $\dot{\varphi}(t) = -F\big(\varphi (t)\big), \ t >0,$ with $\varphi (0+) = \infty$ is called the relaxation function to $F$. 
\end{definition}

In the following, we will have a special focus on the condition $CD_\Upsilon (0,F)$, which holds if for any $u\in \ell^\infty(\mathbb{Z})$ and every $x\in \iZ$ we have
\begin{align*}
\Psi_{2,\Upsilon} (u)(x) \geq F\big(-Lu(x)\big)\quad \mbox{whenever}\; -Lu(x) \geq 0,
\end{align*}
where $F$ is a $CD$-function and $\Psi_{2,\Upsilon}$ is as in (\ref{Psi_2H Alternativ}). This condition will play a key role when proving a Li-Yau type estimate in Section \ref{sec:LiYau}. Recall that the function $\Upsilon$ is defined by
\begin{align*}
\Upsilon (x) = e^x-1-x, \ x \in \R.
\end{align*}
An important assumption that we will make throughout our proofs is justified by the following lemma.

\begin{lemma}
\label{Lemma u(0)=0}
When studying the $CD_\Upsilon(0,F)$ condition for $u\in \ell^\infty(\mathbb{Z})$ at the point $x \in \mathbb{Z}$ we may always assume without loss of generality that $x=0$ and $u(0)=0$.
\end{lemma}

\begin{proof}
Let $x \in \mathbb{Z}$ and $u \in \ell^\infty(\mathbb{Z})$. For $\tilde{u} (y) := u(x+y)-u(x)$ we find that $\tilde{u} (0) = 0$. Furthermore, (\ref{Psi_2H Alternativ}) yields
\begin{align*}
2 &\Psi_{2,\Upsilon}(u)(x) = \sum_{j,l \in \mathbb{Z}} k(j) k(l) e^{u(x+j)-u(x)} \Upsilon(u(x+j+l)-u(x+j)-u(x+l)+u(x)\big) \\ &= \sum_{j,l \in \mathbb{Z}} k(j) k(l) e^{\tilde{u}(j)} \Upsilon\Big(\big(u(x+j+l)-u(x)\big)-\big(u(x+j)-u(x)\big)-\big(u(x+l)-u(x)\big)\Big) \\ &= \sum_{j,l \in \mathbb{Z}} k(j) k(l) e^{\tilde{u}(j)} \Upsilon\big(\tilde{u}(j+l)-\tilde{u}(j)-\tilde{u}(l)\big) \\ &= 2 \Psi_{2,\Upsilon}(\tilde{u})(0).
\end{align*}
As we also have $Lu(x) = L \tilde{u} (0)$ and therefore $F\big(-Lu(x)\big) = F\big(-L \tilde{u}(0)\big)$ the assertation of the lemma follows.
\end{proof}

Throughout this article we denote by $F_\gamma:\mathbb{R} \to \mathbb{R}$ the function given by 
\begin{equation}
F_\gamma (x) = \vert x \vert ^\gamma 
\label{Definition F}
\end{equation}
with $\gamma \in [2,\infty)$. Note that $F_\gamma \vert_{[0,\infty)}$ is a $CD$-function.

The following lemma provides a useful upper estimate for $F_\gamma$ we will employ frequently.

\begin{lemma}
\label{Abschaetzung F Phi}
Define the function $H: \mathbb{R} \times \mathbb{R} \to \mathbb{R}$ by $H (x,y) = \frac{1}{2} e^{-x} \Upsilon(x+y)$. 
Let $\gamma \geq 2$. Then there exists a constant $\nu > 0$ depending only on $\gamma$ such that
\begin{equation*}
F_\gamma (y) \leq \nu H(x,y),\quad  x,\,y \in [0,\infty).
\end{equation*}
\end{lemma}

\begin{proof}
We have $H (x,y) = \frac{1}{2} \big(e^y -(x+y+1)e^{-x}\big)$. Evidently, there exists a $\nu > 0$ such that
\begin{align*}
H(0,y) - \nu F_\gamma (y) = \frac{1}{2} \Upsilon (y) - \nu F_\gamma (y) >0,\quad y\ge 0. 
\end{align*}
Further, for any $y \geq 0$ we have
\begin{equation*}
\frac{\partial}{\partial x} \big(H(x,y) - \nu F_\gamma (y)\big) = \frac{1}{2} (x+y)e^{-x} >0, \ x > 0,
\end{equation*}
which yields the statement.
\end{proof}

In what follows we will also write $H(x) := H(x,x)$, i.e. 
\begin{equation}\label{functionF}
H(x)=\frac{1}{2}e^{-x}\Upsilon(2x), \quad x \in \R.
\end{equation}
This function will play a key role in our calculations.

\begin{bemerkung}
\label{BemerkungH}
Let $\gamma \geq 2$. There exists a constant $\nu=\nu(\gamma) >0$ such that
\begin{align} \label{FgammaH}
F_\gamma (x) \leq \frac{\nu}{2} e^{-x} \Upsilon(2x) = \nu H(x), \ x \in \iR.
\end{align}
This follows directly from the fact that the continuous and positive function $H$ behaves quadratically near its only zero $x=0$ and grows exponentially as $x\to \pm \infty$. We also note that \eqref{FgammaH} holds with the constant $\nu$ from
Lemma \ref{Abschaetzung F Phi}. This is clear for all $x\ge 0$. The case $x<0$ can be reduced to the first case by means of symmetry of $F_\gamma$ and the inequality $H(x)\le H(-x)$, $x\ge 0$, which is not difficult to see.
\end{bemerkung}

In order to describe the asymptotic behaviour of positive functions we will often use the notation
\begin{align*}
f(x) \sim g(x) \text{ as } x \to x_0
\end{align*}
for $x_0 \in \R \cup \{\infty\}$ if $\frac{f(x)}{g(x)} \to 1$ as $x \to x_0$. 

We next collect some asymptotic properties of the function $H$ and its derivative.

\begin{lemma} \label{Falphaprop}
The function $H':\iR \to \iR$ is a strictly increasing bijection that is given by
\begin{align*}
H'(x) & = \frac{1}{2}\big(e^x - e^{-x} + 2xe^{-x}\big).
\end{align*}
In particular, $H'(0)=0$. Moreover, $H$ and $H'$ enjoy
the subsequent asymptotic properties:
\begin{equation} \label{asympF}
H(x) \sim  x^2 \mbox{ as}\;x\to 0 \,\,  
  \mbox{and}\,\, H(x) \sim \frac{e^x}{2} \mbox{ as}\;x\to \infty,
\end{equation}
\begin{equation}  \label{asympDF}
H'(x)\sim 2 x \mbox{ as}\;x\to 0 \,\,  
 \mbox{and}\,\, H'(x)\sim \frac{e^{x}}{2} \mbox{ as}\;x\to \infty,
\end{equation}
\begin{equation}  \label{asympDF-}
(H')^{-1}(x) \sim \frac{x}{2} \mbox{ as}\;x\to 0 \,  
 \, \mbox{and}\,\, (H')^{-1}(x) \sim \log x
 \mbox{ as}\;x\to \infty
\end{equation}
and also
\begin{equation}\label{asympH-1add}
(H')^{-1}(x) - \log(2x) \to 0 \text{ as } x \to \infty.
\end{equation}
\end{lemma}
\begin{proof}
The relations \eqref{asympF} and \eqref{asympDF} are straightforward. The first part of \eqref{asympDF-} follows directly from \eqref{asympDF} by substituting $y=(H')^{-1}(x)$ and observing that
\begin{align*}
\frac{2 (H')^{-1}(x)}{x}= \frac{2y}{H'(y)}.
\end{align*} 
The second statement in \eqref{asympDF-} is a consequence of \eqref{asympH-1add}. To see the latter we again substitute $y=(H')^{-1}(x)$ and write
\begin{equation*}
(H')^{-1}(x)- \log(2x) = y - \log(2 H'(y)) = - \log(2 H'(y)e^{-y}),
\end{equation*}
from which we obtain \eqref{asympH-1add} by using the second part of \eqref{asympDF}.
\end{proof}

We now turn to the consideration of the $CD_\Upsilon (0,F)$ condition. We remind the reader of our standing hypothesis \eqref{StandingHyp}.

For a given  $u:\iZ \to \iR$, we define the symmetric function
\begin{equation}
w(j) := \frac{u(j)+u(-j)}{2}, \ j \in \mathbb{Z}.
\label{w(j)}
\end{equation}
Regarding the operator $L$, under the assumption that $u(0)=0$ we see that at $x=0$
\begin{equation}
Lw(0) = \sum_{j \in \mathbb{Z}} k(j) \big(w(j)-w(0)\big) = \frac{1}{2} \sum_{j \in \mathbb{Z}} k(j) \big(u(j)+u(-j)\big) = \sum_{j \in \mathbb{Z}} k(j) u(j) = Lu(0),
\label{Gleichheit Lw und Lu}
\end{equation}
by symmetry of the kernel $k$. 

We next derive an important lower estimate for $\Psi_{2,\Upsilon}(u)(0)$ in terms of $w(j)$. Assuming as before that $u(0)=0$ (cf.\ Lemma \ref{Lemma u(0)=0}) we first observe that
\begin{align}
\Psi_{2,\Upsilon} (u)(0) &= \frac{1}{2} \sum_{j,l \in \mathbb{Z}} k(j) k(l) e^{u(j)} \Upsilon\big(u(j+l)-u(j)-u(l)\big) \notag \\ &= \frac{1}{2} \sum_{j,l \in \mathbb{Z}} k(j) k(l) e^{u(l)} \Upsilon\big(u(j+l)-u(j)-u(l)\big).
\label{Psi_2Ups Vertauschung l und j}
\end{align}
From this and the convexity of the exponential function we obtain
\begin{align}
\Psi_{2,\Upsilon} (u)(0) &= \frac{1}{2} \sum_{j,l \in \mathbb{Z}} k(j) k(l) \Big(\frac{e^{u(j)} +e^{u(l)}}{2} \Big)\Upsilon\big(u(j+l)-u(j)-u(l)\big) \notag \\ &\geq \frac{1}{2} \sum_{j,l \in \mathbb{Z}} k(j) k(l) e^{\frac{u(j)+u(l)}{2}}\Upsilon\big(u(j+l)-u(j)-u(l)\big) \notag \\ &\geq \frac{1}{2} \sum_{j \in \mathbb{Z}} k(j)^2 e^{w(j)}\Upsilon\big(-2w(j)\big) = \sum_{j \in \mathbb{Z}} k(j)^2 H \big(-w(j)\big),
\label{Basic estimate Psi_2Palme}
\end{align}
where, in the last step, we estimated the $l$-sum from below by dropping all summands with $l\neq -j$.
 
As in the case of the Bakry-\'Emery condition $CD(0,n)$ (see \cite[Theorem 2.1]{SWZ1}), the basic estimate \eqref{Basic estimate Psi_2Palme} is sufficient to show that kernels of finite support satisfy $CD_\Upsilon(0,F)$ with some explicit $CD$-function.

\begin{korollar}\label{Finitesupportresult}
Let $L$ be as in (\ref{eq:operator}). If the kernel associated with $L$ has finite support, then $L$ satisfies $CD_\Upsilon(0,F)$ with the $CD$-function $F(r)=c_0 |k|_1 H\big(\frac{r}{|k|_1}\big)$, where $c_0 = \min\limits_{j \in \mathrm{supp}(k)}k(j)>0$.
\end{korollar}
\begin{proof}
Thanks to Lemma \ref{Lemma u(0)=0} we may restrict ourselves to the situation where $u(0)=0$ and $x=0$.
We define $S= \mathrm{supp}(k)$ and employ estimate \eqref{Basic estimate Psi_2Palme} to observe that
\begin{align*}
\Psi_{2,\Upsilon}(u)(0) &\geq \frac{c_0}{2}\sum\limits_{j \in S} k(j) e^{w(j)}\Upsilon \big( -2 w(j)\big)\\
&\geq F\Big(- \sum\limits_{j \in S} k(j)w(j)\Big) =  F\big(- Lu (0)\big),
\end{align*}
where we applied Jensen's inequality (with the convex mapping $r\mapsto H(-r)$, $r\in \iR$) in the second to last step.
\end{proof}
The following example shows that the latter result is sharp. It has already been used for observing sharpness of the corresponding result regarding the Bakry-\'Emery condition for kernels with finite support, see \cite[Example 2.1]{SWZ1}.
\begin{beispiel}
For $N \in \iN$ we consider
\[
k(x) = \begin{cases}
1, & x = 2j+1 \text{ for some } -N \leq j\leq N-1\\
0,& \text{else}
\end{cases}
\quad \text{ and } \quad
u(x) = \begin{cases}
-1,& x \text{ is odd}\\
0, & x = 0\\
-2, &\text{else.}
\end{cases}
\]
Clearly, $k$ is symmetric and has finite support. Note that $u(2j + 2l +2) \neq -2$ if and only if $l =
-j-1$, which yields
\begin{align*}
\Psi_{2, \Upsilon}(u)(0) &= \frac{1}{2}\sum\limits_{j,l=-N}^{N-1} e^{-1} \Upsilon\big(u(2j + 2l +2)+1+1\big)= \frac{1}{2}\sum\limits_{j=-N}^{N-1} \frac{\Upsilon(2)}{e} = \frac{N\,\Upsilon(2)}{e}.
\end{align*}
Further, we observe that 
\begin{equation*}
Lu (0) = \sum_{j=-N}^{N-1} k(2j+1)u(2j+1)=-2N.
\end{equation*}
Using that $|k|_1=2N$ and $c_0=1$ (where $c_0$ is defined as in Corollary \ref{Finitesupportresult}), we conclude that
\begin{equation*}
F\big(-L u(0)\big) = \frac{N\,\Upsilon(2)}{e} =  \Psi_{2, \Upsilon}(u)(0),
\end{equation*}
where $F$ denotes the $CD$-function of Corollary \ref{Finitesupportresult}.
\end{beispiel}

We now give a simple but quite important result for the case of general kernels.

\begin{satz}\label{basicCDresult}
Let the operator $L$ be given by (\ref{eq:operator}). Suppose that the kernel associated with $L$ satisfies $\hat{C}_\delta(k):=\sum_{j \in \mathbb{Z}} k(j)^{1-\delta} < \infty$ for some $\delta \in (0,1)$. Then, for every $x\in \iZ$ and any
$u\in \ell^\infty(\mathbb{Z})$ there holds
\begin{equation*}
\Psi_{2,\Upsilon}(u)(x) \geq c\big \vert Lu(x)\big \vert ^\gamma,
\end{equation*}
where $\gamma = {\frac{1+\delta}{\delta}}>2$ and $c>0$ is some constant. In particular, 
$L$ satisfies the condition $CD_\Upsilon(0,F)$ with the $CD$-function $F=c F_\gamma$.
\end{satz}
\begin{proof}
We may again restrict ourselves to the situation where $u(0)=0$ and $x=0$.
Applying Jensen's inequality to the function $F_\gamma(r)=|r|^\gamma$, $r\in \iR$, we have
\begin{align*}
F_\gamma\big(L u(0)\big) = F_\gamma \big(L w (0)\big) &= F_\gamma\Big( \sum_{j \in \mathbb{Z}} \frac{k(j)^{1-\delta}\hat{C}_\delta(k)k(j)^\delta w(j)}{\hat{C}_\delta(k)}\Big) \\
&\leq \frac{1}{\hat{C}_\delta(k)}\sum_{j \in \mathbb{Z}} k(j)^{1-\delta} F_\gamma\big( \hat{C}_\delta(k) k(j)^\delta w(j)\big) \\
&= \hat{C}_\delta(k)^\frac{1}{\delta} \sum_{j \in \mathbb{Z}} k(j)^2 F_\gamma\big(w(j)\big) \leq \hat{C}_\delta(k)^\frac{1}{\delta} \nu \Psi_{2,\Upsilon} (u) (0),
\end{align*}
where the last estimate follows from inequality \eqref{FgammaH} (which applies as $\gamma > 2$) and the basic estimate (\ref{Basic estimate Psi_2Palme}). This proves the claimed inequality with $c=\hat{C}_\delta(k)^{-\frac{1}{\delta}} \nu^{-1}$.
\end{proof}
\begin{bemerkung}\label{rem2basicCDresult}
(i) In the subsequent sections we will improve this result for a large class of kernels, see Theorem \ref{Satz alphates Moment}. Note however, that Theorem \ref{basicCDresult} does not require any monotonicity assumptions.  This is quite remarkable having the negative criterion from \cite[Theorem 5.3]{SWZ1} in mind, which says that $CD(0,d)$ is violated
for all $d<\infty$ if the support of the kernel is sufficiently thin. Consequently, Theorem \ref{basicCDresult} demonstrates that it is useful to consider other $CD$-functions than only the quadratic one.

(ii) Theorem \ref{basicCDresult} applies to the important example of the power-type kernel (see \eqref{powerkernel}). In fact, choosing $\delta=\frac{\alpha}{1+\beta}$ with $0<\alpha<\beta$, the condition $\hat{C}_\delta(k)<\infty$ of Theorem \ref{basicCDresult} is equivalent to the power-type kernel having finite $\alpha$-th moment, which holds since $\alpha<\beta$.
\end{bemerkung}
\section{Optimizing the $CD$-function from the basic estimate}
\label{Chapter3}
In this section we determine the optimal $CD$-function one can get from the basic estimate \eqref{Basic estimate Psi_2Palme} alone, i.e.\ we aim to find the best possible $CD$-function $G:[0,\infty)\to [0,\infty)$ such that
\begin{equation}\label{estimatetooptimize}
\sum\limits_{j \in \Z} k(j)^2 H(-w(j)) \geq G \big(-L w(0)\big)
\end{equation}
if $- L w (0) \geq 0$.
We will use the notation $S:= \mathrm{supp}(k)$ throughout this section.

Primarily, we are interested in the behaviour of $G$ for large arguments. For small arguments, we will be able to show in Theorem \ref{Satz alphates Moment} that for a large class of kernels the $CD$-function behaves quadratically near zero (which is the best possible power-type behaviour, cf. \cite[Remark 2.4]{WEB}). However, a quadratic behaviour near zero
cannot be reached from the basic estimate \eqref{Basic estimate Psi_2Palme} alone if $k$ has unbounded support. 

Indeed, for bounded $u$ and the corresponding symmetrization $w$ (cf. \eqref{w(j)}) we multiply \eqref{estimatetooptimize} with $\lambda^{-2}$ and replace $u$ by $\lambda u$, $\lambda>0$. Sending $\lambda\to 0$ yields 
\begin{equation}\label{necconditionQbehavior}
\sum\limits_{j \in \Z} k(j)^2w(j)^2 \geq C \Big( \sum\limits_{j \in \Z} k(j)w(j)\Big)^2
\end{equation}
as a necessary condition for the validity of \eqref{estimatetooptimize} when assuming for contradiction that $G(r)\sim C r^2$ as $r\to 0$, where $C>0$ is some constant. Assuming that the support of $k$ is not finite, we can  write $S\cap \N=\{s(j):j \in \N\}$, where $s:\N \to \N$ is some strictly increasing mapping. We denote by $s^{-1}:s(\N)\to \N$ the inverse mapping of $s$. For $N \in \N$ and $j \in \N$ we define
\begin{equation*}
u_N(j) = \left\{\begin{array}{ll}
\frac{1}{k(j)s^{-1}(j)} ,& j\in S \cap \mathbb{N}, j \leq s(N),\\
0,& \text{else }
\end{array}\right.,
\end{equation*}
set $u_N(0)=0$ and extend $u_N$ to a symmetric mapping on $\Z$. Denoting by $w_N$ the respective symmetrization, we hence have $u_N=w_N$ and observe that
\begin{equation*}
\sum\limits_{j \in \Z} k(j)w_N(j) = \sum\limits_{j \in S, |j|\leq s(N)} \frac{1}{s^{-1}(|j|)} =2 \sum\limits_{j=1}^N \frac{1}{j}.
\end{equation*}
Moreover, we have
\begin{equation*}
\sum\limits_{j \in \Z} k(j)^2w_N(j)^2 = \sum\limits_{j \in S, |j|\leq s(N)} \frac{1}{s^{-1}(|j|)^2} = 2 \sum\limits_{j=1}^N \frac{1}{j^2}.
\end{equation*}
Sending $N\to \infty$ hence yields a contradiction since the left-hand side of \eqref{necconditionQbehavior} converges while the right-hand side is unbounded.

After these preliminary thoughts about natural limitations of  inequality \eqref{estimatetooptimize} we now come back to the aim of this section. First, we examine the case where the kernel $k$ has  finite support and then extend the respective result by approximation to more general kernels. 

We suppose first that the support $S$ is a finite set. Let $X=\{v:S \to \R\}$ be the space of real-valued functions defined on $S$. On $X$ we define the functionals
\[
\Phi(v)=\sum_{j\in S } k(j)^2 H\big(v_j\big)\quad \mbox{and}\quad
\Lambda(v)=\sum_{j\in S } k(j) v_j, \quad v=(v_j)_{j\in S}.
\]
Let $a>0$ be arbitrarily fixed and consider on $X$ the minimization problem
\begin{equation} \label{minproblem}
\Phi(v) \rightarrow \mbox{Min!}\quad \mbox{under the constraint}\;\Lambda(v)=a.
\end{equation}
We will show that \eqref{minproblem} has a unique (global) minimizer $v(a)$. Defining then $G:(0,\infty) \to [0,\infty)$ by
\[
G(a):=\Phi(v(a)), \quad a>0,
\]
it follows that \eqref{estimatetooptimize} is satisfied for this $G$.


The constrained minimization problem \eqref{minproblem} can  be solved  by the method of Lagrange multipliers. The Fr\'echet
derivatives of $\Phi$ and $\Lambda$, respectively, at $v\in X$ are given by
\[
\Phi'(v)\bar{v}=\sum_{j\in S} k(j)^2 H'\big({v}_j\big)\bar{v}_j
\;\; \mbox{and}\;\; \Lambda'(v)\bar{v}=\Lambda(\bar{v})=\sum_{j\in S} k(j) \bar{v}_j,\;
\bar{v}=(\bar{v}_j)\in X.
\]
Clearly $\Lambda'(v)\neq 0$ for all $v\in X$ and thus 
$v_*=(v_j)\in X$ is a critical point for the minimization problem \eqref{minproblem} if and only if 
$\Lambda(v_*)=a$ and there exists a number $\lambda\in \iR$ such that
\[
\Phi'(v_*)=\lambda \Lambda'(v_*),
\] 
that is
\[
\sum_{j\in S} k(j)^2 H'\big(v_j\big)\bar{v}_j=\lambda \sum_{j\in S} k(j) \bar{v}_j,\quad
\mbox{for all}\;\bar{v}=(\bar{v}_j)\in X.
\]
Taking for $\bar{v}$ the unit vectors in $X$, we infer that
\[
k(j) H'(v_j)=\lambda,\quad j\in S.
\]
By Lemma \ref{Falphaprop}
it follows that
\begin{equation} \label{wjformula}
v_j=(H')^{-1}\Big(\frac{\lambda}{k(j)}\Big),\quad j\in S.
\end{equation}
Since $H'(0)=0$ and $H'$ is strictly increasing, the inequality $\lambda\le 0$ would
imply that $v_j\le 0$ for all $j\in S$, which is impossible in view of the constraint
$\Lambda(v)=a>0$. Therefore, we can assume that $\lambda>0$.
Inserting \eqref{wjformula} into the constraint $\Lambda(v)=a$ we obtain
\begin{equation} \label{rhodef}
\rho(\lambda):= \sum_{j\in S}
 k(j) (H')^{-1}\Big(\frac{\lambda}{k(j)}\Big)=a.
\end{equation}
Recall that $H'(0)=0$ and $H'$ is a strictly increasing bijection, see Lemma \ref{Falphaprop}. Therefore
$\rho$ is a strictly increasing bijection from $[0,\infty)$ onto itself, where $\rho(0)=0$.
Consequently, \eqref{rhodef} possesses the unique solution 
\[
\lambda=\rho^{-1}(a),
\]
that is, the Lagrange multiplier $\lambda$ is uniquely determined. Inserting this into \eqref{wjformula} we see that there is only one critical point $v_*=(v_j)\in X$, which is given by 
\begin{equation} \label{wjexplicit}
v_j=(H')^{-1}\Big(\frac{\rho^{-1}(a)}{k(j)}\Big),\quad j\in S.
\end{equation}

Moreover, the set of all $v\in X$ for which $\Lambda(v)=a$ holds is a convex set, by linearity of $\Lambda$.
In addition, $\Phi$ is convex, even strict convex, since $H$ possesses this property. Invoking the
theory of convex optimization, it follows that the only critical point $v_*$ is the (global) minimizer
for the extremal problem \eqref{minproblem}. Hence we obtain the desired functional inequality
\eqref{estimatetooptimize} with the optimal (!) function
\begin{equation} \label{Gfunction}
G(a)=\sum_{j\in S} k(j)^2 H\Big( (H')^{-1}\Big(\frac{\rho^{-1}(a)}{k(j)}\Big) \Big),\quad a>0.
\end{equation}
In particular, we can extend $G$ onto $[0,\infty)$ continuously by setting $G(0)=0$.

Differentiating $G$ yields a much simpler representation formula for $G$. In order to derive it, we first calculate the derivative of $\rho$. We have
\begin{equation*}
\rho'(\lambda)= \sum\limits_{j \in S} \big( (H')^{-1}\big)'\Big( \frac{\lambda}{k(j)}\Big) 
\end{equation*}
and hence also 
\begin{equation}\label{inverserhoDerivative}
(\rho^{-1})'(a) = \frac{1}{\sum\limits_{j \in S} \big( (H')^{-1}\big)'\Big( \frac{\rho^{-1}(a)}{k(j)}\Big)}.
\end{equation}
Then, the derivative of $G$ can be calculated as follows
\begin{align*}
G'(a) &= \sum\limits_{j \in S} k(j)^2 \frac{d}{da} H\Big( (H')^{-1}\Big(\frac{\rho^{-1}(a)}{k(j)}\Big) \Big)\\
&= \sum\limits_{j \in S} k(j) H'\Big( (H')^{-1}\Big( \frac{\rho^{-1}(a)}{k(j)}\Big)\Big) \big( (H')^{-1}\big)'\Big(\frac{\rho^{-1}(a)}{k(j)}\Big) \big(\rho^{-1}\big)'(a) \\
&= \rho^{-1}(a)\big(\rho^{-1}\big)'(a)\, \sum\limits_{j \in S} \big( (H')^{-1}\big)'\Big(\frac{\rho^{-1}(a)}{k(j)}\Big)\\
&= \rho^{-1}(a),
 \end{align*}
where we applied \eqref{inverserhoDerivative} in the last step. Since $G(0)=0$, we thus observe that
\begin{equation}\label{intformulaG}
G(a)= \int_0^a \rho^{-1}(s) \mathrm{d}s.
\end{equation}


We consider now the general case where the support $S$ of $k$ can be infinite. Under an additional assumption
on the kernel, we will reduce this case to the case of finite support by looking at truncated sums and taking limits
in the final step. 

For $v\in \ell_k^1(\mathbb{Z})$ with $\sum_{j\in \iZ} k(j)v(j)>0$, we define
\begin{equation*}
A=\,\sum_{j\in \iZ} k(j)v_j.
\end{equation*}
Setting $\iZ_N=\{j\in \iZ:\,|j|\le N\}$ for $N\in \iN$, it follows that there exists some $N_0 \in \N$ such that
\begin{equation} \label{truncpos}
A_N:=\,\sum_{j\in \iZ_N} k(j)v_j >0
\end{equation}
whenever $N\geq N_0$. For any $N\geq N_0$, we define $S_N := \Z_N \cap S$ 
and
\[
\rho_N(\lambda)= \sum_{j\in S_N}
 k(j) (H')^{-1}\Big(\frac{\lambda}{k(j)}\Big),\quad \lambda>0,
\]
as well as
\begin{equation} \label{GNDef}
G_N(a)=\sum_{j\in S_N} k(j)^2 H\Big( (H')^{-1}\Big(\frac{\rho_N^{-1}(a)}{k(j)}\Big) \Big) = \int_0^a \rho_N^{-1}(s)\mathrm{d}s,\quad a>0.
\end{equation}
From what we know in the case of finite support and by \eqref{truncpos}, it follows that $G_N$ is well-defined and that
\begin{equation} \label{truncest1}
\sum_{j\in \iZ_N} k(j)^2 H(v_j)\ge 
G_N\Big(\sum_{j\in \iZ_N} k(j) v_j\Big)=G_N(A_N).
\end{equation}
By positivity of $H$, we can send $N\to \infty$ on the left-hand side of the inequality \eqref{truncest1}, thereby
obtaining that
\begin{equation} \label{truncest2}
\sum_{j\in \iZ} k(j)^2 H\big(v_j\big)\ge 
G_N(A_N).
\end{equation}

Next, we aim to define the function $\rho$ on $[0,\infty)$ as in  \eqref{rhodef}. Clearly, we then have $\rho(0)=0$. Now let $\lambda>0$.
By the integrability of $k$ we have $k(j)\to 0$ as $|j|\to \infty$. So the argument $\lambda/k(j)$ becomes large for
large $|j|$. Using the asymptotic properties of $(H')^{-1}$
we can prove the following auxiliary result which shows, in particular, that $\rho$ is well-defined.
\begin{lemma}  \label{rhoproperties}
Suppose that the kernel $k$ satisfies the additional condition
\begin{equation} \label{conditionlog}
\sum_{j\in S} k(j) \log\big(2+\frac{1}{k(j)}\big)<\infty.
\end{equation}
Then $\rho:[0,\infty) \to [0,\infty)$, defined as
\begin{equation*}
\rho(\lambda)= \sum_{j\in S}
 k(j) (H')^{-1}\Big(\frac{\lambda}{k(j)}\Big),
\end{equation*}
if $\lambda>0$ and $\rho(0)=0$, is a strictly increasing and continuous bijection with the asymptotic behaviour
\begin{equation}\label{additiveasymprho}
\rho(\lambda)-|k|_1 \log(2\lambda) \to \sum\limits_{j \in S} k(j)\log(\frac{1}{k(j)}\big) \text{ as } \lambda \to \infty.
\end{equation}
Moreover, there exist constants $C_1, C_2>0$
such that 
\begin{equation} \label{rhopropest}
C_1 \big(\bar{\rho}_1(\lambda)+\bar{\rho}_2(\lambda)\big) \le \rho(\lambda) \le  C_2 \big(\bar{\rho}_1(\lambda)+\bar{\rho}_2(\lambda)\big),\quad \lambda>0,
\end{equation}
where
\begin{equation*}
\bar{\rho}_1(\lambda) = 
\sum_{j\in S,\,k(j)<\lambda} k(j) \log\Big(1+\frac{\lambda}{k(j)}\Big)
\end{equation*}
and
\begin{equation*}
\bar{\rho}_2(\lambda)= \sum_{j\in \iZ,\,k(j)\ge \lambda} \lambda .
\end{equation*}
\end{lemma}
\begin{proof}
The property that $\rho$ is strictly increasing and continuous is inherited from $(H')^{-1}$. Condition \eqref{conditionlog} and the dominated convergence theorem together imply that $\rho(\lambda) \to 0$ as $\lambda \to 0^+$. Hence, that $\rho$ is a bijection will follow from the asymptotic behaviour \eqref{additiveasymprho}, which we will now derive. We observe that
\begin{align*}
\rho(\lambda)- &|k|_1 \log(2\lambda) = \sum\limits_{j \in S} k(j) \Big( (H')^{-1}\big( \frac{\lambda}{k(j)}\big) - \log(2\lambda)\Big)\\
&= \sum\limits_{j \in S} k(j) \Big( (H')^{-1}\big( \frac{\lambda}{k(j)}\big) - \log\big( \frac{2\lambda}{k(j)}\big)\Big) + \sum\limits_{j \in S}k(j)\Big( \log \big( \frac{2\lambda}{k(j)}\big)-\log(2\lambda)\Big)\\
&= \sum\limits_{j \in S} k(j) \Big( (H')^{-1}\big( \frac{\lambda}{k(j)}\big) - \log\big( \frac{2\lambda}{k(j)}\big)\Big) + \sum\limits_{j \in S} k(j) \log \big(\frac{1}{k(j)}\big).
\end{align*}
With this at hand and condition \eqref{conditionlog}, we can employ the dominated convergence theorem in order to deduce \eqref{additiveasymprho} from \eqref{asympH-1add}.

Now we turn to the estimates \eqref{rhopropest}. Note that by the assumptions on the kernel we have $\bar{\rho}_i(\lambda)<\infty$ for all $\lambda>0$ and $i=1,2$. We write
\begin{align*}
\rho(\lambda) & =\sum_{j\in S,\,k(j)<\lambda} k(j) (H')^{-1}\Big(\frac{\lambda}{k(j)}\Big)
+\sum_{j\in S,\,k(j)\ge \lambda} k(j) (H')^{-1}\Big(\frac{\lambda}{k(j)}\Big)=:\tilde{\rho}_1(\lambda)
+\tilde{\rho}_2(\lambda).
\end{align*}
The function $(H')^{-1}$ is continuous and maps $(0,\infty)$ onto itself. In view of \eqref{asympDF-} in Lemma \ref{Falphaprop}, it follows that there exist constants $C_1, C_2>0$  such that
\[
C_1\le \frac{(H')^{-1}(x)}{x}\le C_2, \;\; x\in (0,1], \quad 
C_1\le \frac{(H')^{-1}(x)}{\log (1+x)}\le C_2,\;\;x\in [1,\infty). 
\]
These bounds imply that
\[
C_1\le \frac{\tilde{\rho}_i(\lambda)}{\bar{\rho}_i(\lambda)}\le C_2,\quad \lambda>0,\,i=1,2.
\]
This shows the claimed estimates for $\rho(\lambda)$. Further, $\rho$ is a strictly increasing, continuous bijection from $[0,\infty)$ onto itself since $(H')^{-1}$ possesses this property.
\end{proof}
We now return to estimate \eqref{truncest2}. By sending $N\to \infty$ we would like to show that
\begin{equation} \label{aimest}
\sum_{j\in \iZ} k(j)^2 H\big(v_j\big)\ge 
G(A),
\end{equation}
where $G$ is defined by formula \eqref{intformulaG} as $G(a)= \int_0^a \rho^{-1}(s)\mathrm{d}s$ and where we assume that condition \eqref{conditionlog} holds.
From the definition of $\rho_N$ it is clear that $\rho_N\le \rho_{N+1}$ for all $N\geq N_0$. This implies
$\rho_N^{-1}\ge \rho^{-1}_{N+1}$ for all $N\geq N_0$. We also have that  $\lim_{N\to \infty} \rho_N^{-1}(a)=\rho^{-1}(a)$
for all $a  \geq 0$. Using these properties, it follows from the definition of $G_N$ that
\begin{equation*}
\sum_{j\in \iZ} k(j)^2 H\big({v}_j\big) \geq G_N(A_N) = \int_0^{A_N} \rho_N^{-1}(s)\mathrm{d}s \geq \int_0^{A_N}\rho^{-1}(s)\mathrm{d}s = G(A_N)
\end{equation*}
for any $N \geq N_0$. Sending $N\to \infty$ on the right-hand side of the latter estimate yields \eqref{aimest}. We summarize our results in the following theorem.
\begin{satz} \label{thm:Lagrangeestimate}
Suppose that the kernel $k$ satisfies condition \eqref{conditionlog}.
Then for any $v \in \ell_k^1(\mathbb{Z})$ with $\sum_{j\in \iZ} k(j) v_j>0$ there holds
\begin{equation} \label{GInequ}
\sum_{j\in \iZ} k(j)^2 H(v_j)\ge 
G\Big(\sum_{j\in \iZ} k(j) v_j\Big),
\end{equation}
where the function $G$ is given by \eqref{intformulaG}.
\end{satz}
\begin{bemerkung}
The function $G$ given by \eqref{intformulaG} is the best possible function satisfying \eqref{GInequ} for all
$v \in \ell_k^1(\mathbb{Z})$ with $\sum_{j\in \iZ} k(j) v_j>0$. In fact, given $a>0$ define $(v_j)_{j\in \iZ}$
by \eqref{wjexplicit} for $j\in S$ and $v_j=0$ for $j\notin S$. Then $(v_j)_{j\in \iZ}\in \ell_k^1(\mathbb{Z})$ since
$(H')^{-1}(x) \sim \log x$ as $x\to \infty$ (see Lemma \ref{Falphaprop}) and by condition \eqref{conditionlog}.
Further,
\[
\sum_{j\in \iZ} k(j) v_j=\sum_{j\in S} k(j)(H')^{-1}\Big(\frac{\rho^{-1}(a)}{k(j)}\Big)=\rho\big( \rho^{-1}(a)\big)=a,
\]
by the definition of $\rho$ (cf.\ Lemma \ref{rhoproperties}), and 
\begin{align*}
\sum_{j\in \iZ} k(j)^2 & H(v_j) =\sum_{j\in S} k(j)^2 H\Big((H')^{-1}\Big(\frac{\rho^{-1}(a)}{k(j)}\Big) \Big)\\
&=\lim_{N\to \infty}\sum_{j\in S_N} k(j)^2 H\Big((H')^{-1}\Big(\frac{\rho_N^{-1}(a)}{k(j)}\Big) \Big)\\
&=\lim_{N\to \infty} G_N(a)=G(a)=G\big(\sum_{j\in \iZ} k(j) v_j \big), 
\end{align*}
in view of \eqref{GNDef} and the definition of $G$. Hence we have equality in \eqref{GInequ}. Since $a>0$ was arbitrary, the
assertion follows.
\end{bemerkung}
In order to make Theorem \ref{thm:Lagrangeestimate} applicable for our purposes, we need to ensure that $G$ is in fact a $CD$-function. 
\begin{satz}\label{CDresultLagrange}
If the associated kernel $k$ to the Markov generator $L$ satisfies condition \eqref{conditionlog},
then $L$ satisfies $CD_\Upsilon(0,G)$ with the strictly convex $CD$-function $G$ that is given by \eqref{intformulaG} and enjoys the asymptotic property
\begin{equation}\label{Gslargearguments}
G(a) \sim \frac{|k|_1 e^\frac{a-M(k)}{|k|_1}}{2} \text{ as } a \to \infty,
\end{equation}
where $M(k)=\sum\limits_{j \in S} k(j)\log \big( \frac{1}{k(j)}\big)$.
\end{satz}
\begin{proof}
Combining the basic estimate \eqref{Basic estimate Psi_2Palme} with Theorem \ref{thm:Lagrangeestimate}, we only need to show that $G$ is a $CD$-function obeying the mentioned properties. Strict convexity of $G$ follows from the fact that $\rho^{-1}$ is strictly increasing. Consequently, we also have that the mapping $a \mapsto \frac{G(a)}{a}$ is strictly increasing on $(0,\infty)$ (cf. \cite[Remark 3.3]{DKZ}). $G(0)=0$ is clear by definition. It remains to show the asymptotic behaviour displayed in \eqref{Gslargearguments}. We substitute $y= \rho^{-1}(a)$, $a>0$, and observe that
\begin{align*}
\frac{2\rho^{-1}(a)}{\exp\big(\frac{a-M(k)}{|k|_1}\big)} = \frac{2y}{\exp\big(\frac{\rho(y)-M(k)}{|k|_1}\big)}= \exp\Big( \log(2y)-\frac{\rho(y)}{|k|_1}+\frac{M(k)}{|k|_1}\Big).
\end{align*}
Hence, \eqref{additiveasymprho} implies
\begin{equation*}
\rho^{-1}(a) \sim \frac{e^{\frac{a-M(k)}{|k|_1}}}{2} \text{ as } a \to \infty.
\end{equation*}
Let $\varepsilon\in(0,1)$ be fixed. Then there exists $t_0(\varepsilon)>0$ such that 
\begin{equation*}
1-\frac{\varepsilon}{2} \leq \frac{\rho^{-1}(t)}{\frac{e^{\frac{t-M(k)}{|k|_1}}}{2}}\leq 1+\frac{\varepsilon}{2}
\end{equation*}
holds for any $t \geq t_0(\varepsilon)$. We observe for $a \geq t_0(\varepsilon)$ that
\begin{align*}
G(a) = \int_0^a \rho^{-1}(s)\mathrm{d}s \geq \int_{t_0(\varepsilon)}^a \rho^{-1}(s)\mathrm{d}s \geq \frac{1-\frac{\varepsilon}{2}}{2}\int_{t_0(\varepsilon)}^a e^{\frac{s-M(k)}{|k|_1}}\mathrm{d}s = \frac{(1-\frac{\varepsilon}{2})|k|_1}{2e^\frac{M(k)}{|k|_1}}\Big( e^\frac{a}{|k|_1} - e^\frac{t_0(\varepsilon)}{|k|_1}\Big)
\end{align*}
and also
\begin{align*}
G(a) &\leq \int_0^{t_0(\varepsilon)}\rho^{-1}(s)\mathrm{d}s + \frac{1+\frac{\varepsilon}{2}}{2} \int_{t_0(\varepsilon)}^a e^{\frac{s-M(k)}{|k|_1}}\mathrm{d}s \\
& = \int_0^{t_0(\varepsilon)}\rho^{-1}(s)\mathrm{d}s + \frac{(1+\frac{\varepsilon}{2})|k|_1}{2e^\frac{M(k)}{|k|_1}}\Big( e^\frac{a}{|k|_1} - e^\frac{t_0(\varepsilon)}{|k|_1}\Big).
\end{align*}
Consequently, there exists some $a_0 \geq t_0(\varepsilon)$ such that 
\begin{equation*}
\frac{(1-\varepsilon)|k|_1}{2e^\frac{M(k)}{|k|_1}}\leq \frac{G(a)}{e^\frac{a}{|k|_1}}\leq \frac{(1+\varepsilon)|k|_1}{2e^\frac{M(k)}{|k|_1}}
\end{equation*}
holds for any $a\geq a_0$, which shows \eqref{Gslargearguments}.
\end{proof}

Let us illustrate the above findings with the operator $L_\beta$ induced by the power-type kernel. Condition \eqref{conditionlog} is met since the stronger condition of Theorem \ref{basicCDresult} is satisfied, see Remark \ref{rem2basicCDresult}(ii). We now focus on finding suitable estimates for $\rho$. It is straightforward to check that the mapping $y\mapsto \frac{\log(1+y)}{y}$ is decreasing on $(0,\infty)$. This yields, in particular, that the mapping $s \mapsto\frac{\log(1+\lambda s^{1+\beta})}{s^{1+\beta}}$ is also decreasing on $(0,\infty)$, where here and in the sequel $\lambda \le 1$ is assumed.  With this at hand we have
\begin{equation}\label{integralsumestimateforrho1}
\int_{r_+(\lambda)}^\infty \frac{\log\big(1+\lambda s^{1+\beta}\big)}{s^{1+\beta}}\mathrm{d}s \leq \sum\limits_{j=r_+(\lambda)}^\infty \frac{\log\big(1+\lambda j^{1+\beta}\big)}{j^{1+\beta}} \leq \int_{r_-(\lambda)}^\infty \frac{\log\big(1+\lambda s^{1+\beta}\big)}{s^{1+\beta}}\mathrm{d}s,
\end{equation}
where we have used the notation $r_-(\lambda):=\lfloor \lambda^{-\frac{1}{1+\beta}}\rfloor $ and $r_+(\lambda):=\lceil \lambda^{-\frac{1}{1+\beta}}\rceil$. Since the ratios $\frac{r}{\lfloor r \rfloor}$ and $\frac{\lceil r \rceil}{r}$ both converge to $1$ as $r \to \infty$, we can find for given $\delta \in (0,1)$ some $\lambda_0(\delta)>0$ small enough such that
$r_-(\lambda)\geq (1-\delta)\lambda^{-\frac{1}{1+\beta}}$ and $r_+(\lambda) \leq (1+\delta) \lambda^{-\frac{1}{1+\beta}}$ hold for every $\lambda \in (0,\lambda_0(\delta))$. We observe that
\begin{equation*}
\int_{(1\pm \delta) \lambda^{-\frac{1}{1+\beta}}}^\infty \frac{\log(1+\lambda s^{1+\beta)})}{s^{1+\beta}}\mathrm{d}s = \lambda^{\frac{\beta}{1+\beta}} \int_{1\pm \delta}^\infty \frac{\log(1+\sigma^{1+\beta})}{\sigma^{1+\beta}}\mathrm{d}\sigma.
\end{equation*}
Recalling the definition of $\bar{\rho}_1$ from Lemma \ref{rhoproperties}, we can now read from \eqref{integralsumestimateforrho1} that
\begin{equation*}
2\lambda^\frac{\beta}{1+\beta} \int_{1+\delta}^\infty \frac{\log(1+\sigma^{1+\beta})}{\sigma^{1+\beta}}\mathrm{d}\sigma \leq \bar{\rho}_1(\lambda) \leq 2\lambda^\frac{\beta}{1+\beta} \int_{1-\delta}^\infty \frac{\log(1+\sigma^{1+\beta})}{\sigma^{1+\beta}}\mathrm{d}\sigma,
\end{equation*}
for any $\lambda \in (0,\lambda_0(\delta))$.

Regarding the expression $\bar{\rho}_2(\lambda)$ from Lemma \ref{rhoproperties}, we have that
$
\bar{\rho}_2(\lambda)=2 \lambda r_-(\lambda)
$
and consequently 
\begin{equation*}
2(1-\delta)\lambda^\frac{\beta}{1+\beta} \leq \bar{\rho}_2(\lambda) \leq 2(1+\delta)\lambda^\frac{\beta}{1+\beta}
\end{equation*}
holds for any $\lambda \in (0,\lambda_0(\delta))$.

Combining these estimates with Lemma \ref{rhoproperties} yields
\begin{equation*}
2 C_1 \Big( \int_{1+ \delta}^\infty \frac{\log(1+\sigma^{1+\beta})}{\sigma^{1+\beta}}\mathrm{d}\sigma + 1-\delta\Big) \lambda^\frac{\beta}{1+\beta} \leq \rho(\lambda)\leq 2 C_2 \Big( \int_{1- \delta}^\infty \frac{\log(1+\sigma^{1+\beta})}{\sigma^{1+\beta}}\mathrm{d}\sigma + 1+\delta\Big) \lambda^\frac{\beta}{1+\beta}
\end{equation*}
for any $\lambda \in(0,\lambda_0(\delta))$, where the constants $C_1,C_2>0$ come from Lemma \ref{rhoproperties}. Consequently, there exist constants $c(\delta),C(\delta)>0$  such that
\begin{equation*}
c(\delta) a^\frac{1+\beta}{\beta}\leq \rho^{-1}(a)\leq C(\delta) a^\frac{1+\beta}{\beta}
\end{equation*}
holds for any $a \in (0,a_0(\delta))$, where $a_0(\delta):= \rho(\lambda_0(\delta))$.
Estimating $G$ from below by 
\begin{equation*}
G(a) \geq c(\delta) \int_0^a s^\frac{1+\beta}{\beta}\mathrm{d}s = \frac{c(\delta)\beta}{1+\beta} a^\frac{1+2\beta}{\beta}
\end{equation*}
for $a \in (0,a_0(\delta))$ we arrive at the following result. 
\begin{korollar}\label{powertypeCDfrombasicestimate}
Let $\beta \in(0,\infty)$. Then the operator $L_\beta$ satisfies $CD_\Upsilon(0,G_\beta)$ with a strictly convex $CD$-function $G_\beta$ that has the asymptotic behaviour $G_\beta(r)\sim \frac{|k_\beta|_1 e^\frac{r-M(k_\beta)}{|k_\beta|_1}}{2}$ as $r \to \infty$ and for which $G_\beta(r)\geq c r^\frac{1+2\beta}{\beta}$ holds for any $r\ge 0$, with some constant $c>0$. 
\end{korollar}
\begin{bemerkung}
Note that Corollary \ref{powertypeCDfrombasicestimate} improves Theorem \ref{basicCDresult} in case of the power-type kernel not only for large arguments but also for small ones. In fact, due to Remark  \ref{rem2basicCDresult}(ii) the $CD$-function from Theorem \ref{basicCDresult} can be chosen such that it behaves like $r^\frac{1+2\beta-\varepsilon}{\beta-\varepsilon}$ for $\varepsilon>0$ arbitrary small. However, setting $\varepsilon=0$ is not allowed in Theorem \ref{basicCDresult}. Hence Corollary \ref{powertypeCDfrombasicestimate} yields the limiting case of Theorem \ref{basicCDresult} for small arguments in case of the power-type kernel.
\end{bemerkung}
Clearly, the above calculations rely on the specific example of the power-type kernel. But together with Theorem \ref{basicCDresult}, Theorem \ref{CDresultLagrange} yields under quite mild assumptions already a $CD$-function that is strong enough  to imply the applications which will be described in Section \ref{sec:LiYau} and Section \ref{Chapter6}.
\begin{korollar}\label{cor:asymptoticscombined}
If there exists some $\tau \in (0,1)$ such that 
\begin{equation}\label{conditiondelta}
\sum\limits_{j \in \mathbb{Z}} k(j)^{1-\tau} <\infty,
\end{equation}
then the Markov generator associated with $k$ satisfies $CD_\Upsilon(0,F)$ with a strictly convex $CD$-function $F$ enjoying the asymptotic properties
\begin{equation*}
\begin{split}
F(r) &\sim c_1e^{\frac{r}{|k|_1}} \text{ as } r \to \infty,\\
F(r) &\sim c_2r^\frac{1+\tau}{\tau} \text{ as } r \to 0,
\end{split}
\end{equation*}
for some $c_1,c_2>0$. 
\end{korollar}
\begin{proof}
This follows from Theorem \ref{basicCDresult}, Theorem \ref{CDresultLagrange} and the fact that condition \eqref{conditiondelta} implies \eqref{conditionlog}.
\end{proof}
\begin{bemerkung}\label{rem:kernelforcond2compare}
The quite interesting kernel
\begin{equation*}
k(j)= \frac{1}{|j| \big(\log(1+|j|)\big)^{1+\alpha}}, \quad j \in \mathbb{Z}\setminus\{0\},
\end{equation*}
where $\alpha>0$, is integrable but condition \eqref{conditiondelta} is not satisfied. Moreover, we note that for $|j|$ large enough
\[
\log(|j|+1) \le \log\big(2+\frac{1}{k(j)}\big) \leq 3 \log(|j|+1)
\]
%
and hence condition \eqref{conditionlog} is satisfied if and only if $\alpha >1$. In particular, the range of $\alpha \in (0,1]$ yields kernels to which neither Theorem \ref{basicCDresult} nor Theorem \ref{CDresultLagrange} applies.
\end{bemerkung}
What is implicitly used in the proof of Corollary \ref{cor:asymptoticscombined} is that the asymptotic behaviour at $0$ of one $CD$-function and the asymptotic behaviour at $\infty$ from another can be combined to one single $CD$-function that enjoys both asymptotic properties. This is straightforward to see and it will be also used in Section 4, where we will
substantially improve the asymptotic behaviour for small arguments of $CD$-functions for a large class of kernels. We can further bound the respective $CD$-function from below by an explicitly given $CD$-function with the desired asymptotic behaviour. This will be important in Section \ref{sec:LiYau}. More precisely, suppose that the operator $L$ satisfies $CD_\Upsilon (0,F)$ with some $CD$-function $F$ for which there exist constants $c_1,c_2,\delta >0$ and $\gamma\geq 2$ such that
$F(x) \sim c_1 x^\gamma$ as $x \to 0$
and
$
F(x) \sim c_2 e^{\delta x}$ as $x \to \infty$.
Define the function
\begin{equation}\label{eq:Fhat}
\hat{F} (x) = \begin{cases}
e^{\delta M} x^\gamma, x \in [0,M],\\
M^\gamma e^{\delta x}, x > M,
\end{cases}
\end{equation}
where $M > 0$. Then there exists some $c>0$ such that $L$ satisfies $CD_\Upsilon(0,\tilde{F})$ with the $CD$-function $\tilde{F} = c \hat{F}$. Moreover, the constant $M>0$ can be chosen so large that $\tilde{F}$ is convex. We will choose $M\geq \frac{2}{\delta}$ since this guarantees that the mapping $x\mapsto x^{-2}e^{\delta x}$ is increasing for $x>M$. This will be important in Section \ref{sec:LiYau}. The main advantage of this explicit formula for the $CD$-function is that we can derive an explicit representation of the corresponding relaxation function, too.
\begin{lemma}
\label{LemmaRelaxFunc}
Let $\hat{F}$ be given by \eqref{eq:Fhat}, where $\gamma \geq 2, \delta>0$ and $M>0$, and let $\tilde{F}= c\hat{F}$ for some $c>0$. Then the relaxation function corresponding to $\tilde{F}$ is given by 
\begin{align*}
\varphi (t) = \begin{cases}
-\frac{1}{\delta} \log (c \delta M^{\gamma} t), \ t \in (0,t_*],\\
\big(c (\gamma-1)e^{\delta M}t + C \big)^{-\frac{1}{\gamma-1}}, \ t > t_*,
\end{cases}
\end{align*}
where $t_* = \frac{M^{-\gamma}}{c \delta} e^{-\delta M}$ and $C = M^{-\gamma} \big(M-\frac{\gamma-1}{\delta}\big) >0$.
\end{lemma}
\begin{proof}
The relaxation function $\varphi$ corresponding to $\tilde{F}$ is given by $\varphi (t) = G^{-1} (t)$, where  $G(x) = \int_x^\infty \frac{1}{\tilde{F}(r)} \mathrm{d}r$ (see the proof of \cite[Lemma 3.7]{DKZ}). To compute $G$, letting $x \in (0,M]$ we obtain 
\begin{align*}
G(x) &= c^{-1} \Big(e^{-\delta M} \int_x^M r^{-\gamma} \mathrm{d}r +  M^{-\gamma} \int_M^\infty e^{-\delta r} \mathrm{d}r\Big) \\&= c^{-1} \Big(\frac{e^{-\delta M}}{\gamma-1} \big(x^{-\gamma+1}-M^{-\gamma+1}\big) + \frac{M^{-\gamma}}{\delta} e^{-\delta M}\Big).
\end{align*}
For $x>M$ we get
\begin{align*}
G(x) = \frac{M^{-\gamma}}{\delta c} e^{-\delta x}.
\end{align*} 
By positivity of $\tilde{F}$, $G$ is a decreasing function. Thus, we find for any $t \geq G(M) = t_*$
\begin{align*}
\varphi (t) = G^{-1} (t) = \big(c (\gamma-1)e^{\delta M}t + C \big)^{-\frac{1}{\gamma-1}},
\end{align*}
where $C = M^{-\gamma} \Big(M-\frac{\gamma-1}{\delta}\Big) >0$. For $t \in (0,t_*]$ we get
\begin{align*}
\varphi(t) = -\frac{1}{\delta} \log \big(c \delta M^\gamma t\big).
\end{align*}
\end{proof}
\begin{bemerkung}
The findings in \cite[Lemma 3.7]{DKZ} are quite similar to Lemma \ref{LemmaRelaxFunc}. There, instead of a specific $CD$-function, the authors consider $CD$-functions $F$ with $F(x) \sim c e^{\delta x}$ as $x \to \infty$ respectively $F(x) \sim \tilde{c}x^2$ as $x \to 0$ and are only interested in the asymptotic behaviour of the relaxation function. Moreover, they only treat the case $\gamma =2$.
\end{bemerkung}

\begin{bemerkung}
The content of this section can be extended without much effort to the multidimensional setting of $\Z^d$. In this case the lower bound for $G_\beta$ from Corollary \ref{powertypeCDfrombasicestimate} reads as 
$G_\beta(r)\ge c r^\frac{d+2\beta}{\beta}$, $r\ge  0$. However, for the sake of consistency with regard to the following sections we decided to present also this section in the one-dimensional framework.
\end{bemerkung}

\section{Sharper estimate for small arguments}

In the following, we will further improve our already established results. Note that for the fractional discrete Laplacian $-(-\Delta)^\frac{\beta}{2}, \ \beta \in (0,2),$ it would be desirable to prove a $CD_\Upsilon (0,F)$ condition with a $CD$-function $F$ that possesses quadratic behaviour for small arguments if we send $\beta \to 2$. This seems natural as we already know from the continuous setting that the Laplacian satisfies a certain $CD$-inequality with a quadratic $CD$-function, see \cite[Chapter 1.16.2]{BGL}. We also remind the reader of Corollary \ref{Finitesupportresult}, which applies to 
the discrete Laplacian, showing that $CD_\Upsilon (0,F)$ holds with some $CD$-function $F$ that behaves quadratically
near zero. With the findings from the previous sections we have not yet established this desirable limiting behaviour
for the power-type kernel and the fractional discrete Laplacian, respectively, as $\beta\to 2$, but we will succeed in proving it in this section.

In fact, we will obtain this result as a special case of the main theorem of this section, Theorem \ref{Satz alphates Moment},
on non-increasing (on $\iN$) kernels that have finite $\alpha$-th moment for some $\alpha \in (1,2]$ and satisfy a certain
decay condition. Theorem \ref{Satz alphates Moment} includes the important statement that in case $\alpha=2$ (i.e.\
the kernel has finite second moment), the operator $L$ generated by the kernel satisfies $CD_\Upsilon (0,F)$ with a
quadratic $CD$-function. This statement can be seen as an analogue to \cite[Theorem 4.1]{SWZ1}, where the classical
Bakry-\'Emery condition $CD(0,d)$ has been established for such kernels.

For the proof of Theorem \ref{Satz alphates Moment} we will need the following auxiliary result on the function $F_\gamma$ defined in (\ref{Definition F}). 

\begin{lemma}
\label{Lemma 1}
(i) Let $\gamma \geq 2$ and $p,q > 1$ such that $\frac{1}{p}+\frac{1}{q} = 1$. Then for any $x,y \in \R$ we have 
\begin{equation*}
F_\gamma (x+y) \le p^{\gamma-1} F_\gamma (x) + q^{\gamma-1} F_\gamma (y).
\end{equation*}
(ii) Let $\gamma \geq 2$ and $\delta > 0$. Then for any $\lambda \geq 1$ and $x,y \in \mathbb{R}$ we have 
\begin{equation*}
F_\gamma (x+y) \leq \tilde{C} \lambda^{\gamma-1} F_\gamma (x) + \Big(1+\frac{\delta}{\lambda}\Big) F_\gamma (y),
\end{equation*}
where $\tilde{C} = \Big(\frac{(1+\delta)(\gamma-1)}{\delta}\Big)^{\gamma-1}$.
\end{lemma}

\begin{proof}
(i) Since $\gamma \geq 2$ we have that $F_\gamma$ is convex. Thus, we obtain for $x,y \in \mathbb{R}$
\begin{align*}
F_\gamma (x+y) = F_\gamma \Big(\frac{px}{p} + \frac{qy}{q}\Big) \leq \frac{F_\gamma (px)}{p} + \frac{F_\gamma (qy)}{q} = p^{\gamma-1} F_\gamma (x) + q^{\gamma-1} F_\gamma (y).
\end{align*}
(ii) Set $q = \Big(1+\frac{\delta}{\lambda}\Big)^\frac{1}{\gamma-1} >1$ and choose $p>1$ such that $\frac{1}{p}+\frac{1}{q}=1$. Then we have
\begin{align*}
\frac{1}{p} = 1 - \frac{\lambda^\frac{1}{\gamma-1}}{(\lambda+\delta)^\frac{1}{\gamma-1	}} &= \frac{(\lambda+\delta)^\frac{1}{\gamma-1} - \lambda^\frac{1}{\gamma-1}}{(\lambda+\delta)^\frac{1}{\gamma-1}} =\frac{1}{(\lambda+\delta)^\frac{1}{\gamma-1}} \Big(\frac{1}{\gamma-1} \int_\lambda^{\lambda+\delta} t^{\frac{1}{\gamma-1}-1} \ \mathrm{d}t \Big) \\ &\geq \frac{1}{\gamma-1} \frac{(\lambda+\delta)^{\frac{1}{\gamma-1}-1}}{(\lambda+\delta)^\frac{1}{\gamma-1}} \delta = \frac{\delta}{(\gamma-1)(\lambda+\delta)},
\end{align*}
where we used that $\gamma \geq 2$ is equivalent to $\frac{1}{\gamma-1}-1\leq 0$. Consequently, we obtain 
\begin{align*}
p \leq \frac{(\gamma-1)(\lambda+\delta)}{\delta} \leq \lambda \frac{(1+\delta)(\gamma-1)}{\delta} = \lambda \tilde{C}^\frac{1}{\gamma-1}.
\end{align*}
The claim follows with $(i)$.
\end{proof}

We now come to the main theorem of this section. Comparing our situation with the one from the proof 
of \cite[Theorem 4.1]{SWZ1}, one of the main difficulties in proving the desired result in the current setting is the missing symmetry of the function $\Upsilon(x)$ compared to the function $x^2$ that occurred in the proof of \cite{SWZ1}. 
Therefore, we will consider the symmetrized version $w$ which has already been introduced in (\ref{w(j)}).

The theorem reads as follows.
\begin{satz}
\label{Satz alphates Moment}
Let $\alpha \in (1,2]$ and $\gamma = \frac{\alpha}{\alpha-1} \geq 2$, that is $\frac{1}{\alpha}+\frac{1}{\gamma}=1$.
Suppose that the kernel $k$ has finite $\alpha$-th moment, i.e.
\begin{equation}
C_\alpha (k) := \sum_{j \in \mathbb{N}} k(j) j^\alpha < \infty.
\label{Kern Bedingung Moment}
\end{equation}
Assume further that $k$ is non-increasing on $\mathbb{N}$ and that
\begin{equation}
k(j) \leq C k(j+1), \ j \in \mathbb{N},
\label{Kern Bedingung Wachsum}
\end{equation}
for some constant $C \geq 1$.
Then there exists a finite $d>0$ such that for any $x \in \mathbb{Z}$ the operator $L$ associated with $k$ satisfies the inequality
\begin{equation}
\Psi_{2,\Upsilon}(u)(x) \geq \frac{1}{d}\big \vert Lu(x)\big \vert ^\gamma, \ u \in \ell^\infty(\mathbb{Z}). 
\label{Theorem Strong CD-inequ}
\end{equation}
In particular, $L$ satisfies $CD_\Upsilon (0,F)$ with the $CD$-function $F = \frac{1}{d} F_\gamma \vert_{[0,\infty)}$, where $F_\gamma$ is the function defined in (\ref{Definition F}).
\end{satz}
\begin{proof}
The proof contains several steps. The line of arguments is inspired by the proof of \cite[Theorem 4.1]{SWZ1} but needs a lot more work. 

We will use the following notation: For two terms $T(u,k)$ and $\tilde{T}(u,k)$ depending on $u \in \ell^\infty(\mathbb{Z})$ and the kernel $k$ we write $T(u,k) \lesssim \tilde{T}(u,k)$ 
if there exists a constant $C>0$ independent of $u$ such that
$
T(u,k) \leq C  \tilde{T}(u,k) .
$

Let $u \in \ell^\infty (\mathbb{Z})$. Throughout the proof we assume w.l.o.g. that $x=0$ and $u(0)=0$, which is possible by Lemma \ref{Lemma u(0)=0}. In the following calculations we consider the symmetrization $w$ of $u$ defined in (\ref{w(j)}).

\emph{Step 1: Jensen estimate.} Using the symmetry of $w$, the convexity of $F_\gamma$ and the fact that $k$ has finite $\alpha$-th moment, we find with Jensen's inequality and (\ref{Gleichheit Lw und Lu})
\begin{align*}
F_\gamma\big(Lu(0)\big) = F_\gamma\big(Lw(0)\big) &= F_\gamma\Big(2 C_\alpha (k) \sum_{j=1}^\infty \frac{k(j) j^\alpha w(j)}{C_\alpha (k) j^\alpha}\Big) \\ &\leq \frac{2^\gamma}{C_\alpha (k)} \sum_{j=1}^\infty k(j) j^\alpha F_\gamma\Big(\frac{C_\alpha (k) w(j)}{j^\alpha}\Big) \\ &= 2^\gamma C_\alpha (k)^{\gamma-1} \sum_{j=1}^\infty \frac{k(j)}{j^{\alpha(\gamma-1)}} F_\gamma \big(w(j)\big) \\ &= 2^\gamma C_\alpha (k)^{\gamma-1} \sum_{j=1}^\infty \frac{k(j)}{j^\gamma} F_\gamma \big(w(j)\big),
\end{align*}
where we used in the last step that $\gamma = \frac{\alpha}{\alpha-1}$ is equivalent to $\alpha (\gamma-1) = \gamma$.

\emph{Step 2: Basic estimate for $\Psi_{2,\Upsilon}$.} This step is the most demanding part of the proof. We aim to show that 
\begin{equation}
\sum_{j=1}^\infty k(j) F_\gamma \big(w(j+1)-w(j)\big) \lesssim \Psi_{2,\Upsilon} (u)(0).
\label{Step 2 to show}
\end{equation}
Thanks to the $j=1$ summand in (\ref{Basic estimate Psi_2Palme}) and inequality \eqref{FgammaH} ($\nu=\nu(\gamma)$) we have
\begin{equation*}
\Psi_{2,\Upsilon}(u)(0) \geq \frac{1}{2} k(1)^2 e^{w(1)}\Upsilon\big(-2w(1)\big) \ge \frac{k(1)^2}{\nu} F_\gamma \big(w(1)\big),
\end{equation*}
which we will use repeatedly in the sequel. We will also need a few other elementary estimates for $\Psi_{2,\Upsilon}(u)(0)$. By picking the terms with $l=1$ respectively $l=-1$ in (\ref{Psi_2Ups Vertauschung l und j}) we get
\begin{align}
&\Psi_{2,\Upsilon}(u)(0) \geq \frac{1}{2} \sum_{j \in \mathbb{Z}} k(j) k(1) e^{u(1)} \Upsilon\big(u(j+1)-u(j)-u(1)\big)
\label{Estimate 1}
\end{align}
and
\begin{align}
\Psi_{2,\Upsilon}(u)(0) \geq \frac{1}{2} \sum_{j \in \mathbb{Z}} k(j) k(1) e^{u(-1)} \Upsilon\big(u(-(j+1))-u(-j)-u(-1)\big). 
\label{Estimate 2}
\end{align}
Furthermore, with an additional index shift we find
\begin{align}
\Psi_{2,\Upsilon}(u)(0) \ &\geq \ \frac{1}{2} \sum_{j \in \mathbb{Z}} k(j) k(1) e^{u(-1)} \Upsilon\big(u(j-1)-u(j)-u(-1)\big) \notag \\ &=  \ \frac{1}{2} \sum_{j \in \mathbb{Z}} k(j+1) k(1) e^{u(-1)} \Upsilon\big(u(j)-u(j+1)-u(-1)\big)
\label{Estimate 3}
\end{align}
and
\begin{align}
\Psi_{2,\Upsilon}(u)(0) \ &\geq \ \frac{1}{2} \sum_{j \in \mathbb{Z}} k(-j) k(1) e^{u(1)} \Upsilon\big(u(-j+1)-u(-j)-u(1)\big) \notag \\ &= \ \frac{1}{2} \sum_{j \in \mathbb{Z}} k(j+1) k(1) e^{u(1)} \Upsilon\big(u(-j)-u(-(j+1))-u(1)\big).
\label{Estimate 4}
\end{align}
We now define index sets $A_i,B_i \ (i=1,2)$ by
\begin{align*}
&A_1 := \{i \in \mathbb{N}: u(i+1)-u(i) \geq -u(1) \}, \\
&A_2 := \{i \in \mathbb{N}: u(i+1)-u(i) \leq u(-1) \}, \\
&B_1 := \{i \in \mathbb{N}: u(-(i+1))-u(-i) \geq -u(-1) \} \\
\text{and } & B_2 := \{i \in \mathbb{N}: u(-(i+1))-u(-i) \leq u(1) \}.
\end{align*}
Define further $A := A_1 \cup A_2$ and $B := B_1 \cup B_2$. The choice of these sets becomes clear in what follows. First let us note that by Remark \ref{BemerkungH}
and (\ref{Estimate 1}) we have
\begin{align}
\sum_{j \in A_1} & k(j) F_\gamma \big(u(j+1)-u(j)-u(1)\big) \notag \\ &\lesssim \sum_{j \in A_1} k(j) \exp\Big(-\frac{u(j+1)-u(j)-u(1)}{2}\Big) \Upsilon\big(u(j+1)-u(j)-u(1)\big) \notag \\ &\leq \sum_{j \in \mathbb{Z}} k(j) e^{u(1)} \Upsilon\big(u(j+1)-u(j)-u(1)\big) \lesssim \Psi_{2,\Upsilon}(u)(0).
\label{Case 1a}
\end{align} 
Analogeously,
\begin{align}
\sum_{j \in B_1} & k(j) F_\gamma \big(u(-(j+1))-u(-j)-u(-1)\big) \notag \\ &\lesssim \sum_{j \in B_1} k(j) \exp\Big(-\frac{u(-(j+1))-u(-j)-u(-1)}{2}\Big) \Upsilon\big(u(-(j+1))-u(-j)-u(-1)\big) \notag \\ &\leq \sum_{j \in \mathbb{Z}}  k(j) e^{u(-1)} \Upsilon\big(u(-(j+1))-u(-j)-u(-1)\big) \lesssim \Psi_{2,\Upsilon}(u)(0)
\label{Case 2a}
\end{align} 
holds by (\ref{Estimate 2}). Using assumption \eqref{Kern Bedingung Wachsum}, we have
\begin{align}
\sum_{j \in A_2}& k(j) F_\gamma \big(u(j)-u(j+1)-u(-1)\big) \notag \\ &\lesssim \sum_{j \in A_2} k(j+1) \exp\Big(-\frac{u(j)-u(j+1)-u(-1)}{2}\Big) \Upsilon\big(u(j)-u(j+1)-u(-1)\big) \notag \\ &\leq \sum_{j \in \mathbb{Z}} k(j+1) e^{u(-1)} \Upsilon\big(u(j)-u(j+1)-u(-1)\big) \lesssim \Psi_{2,\Upsilon}(u)(0)
\label{Case 1b}
\end{align} 
by (\ref{Estimate 3}), and finally (\ref{Estimate 4}) yields
\begin{align}
\sum_{j \in B_2} & k(j) F_\gamma \big(u(-j)-u(-(j+1))-u(1)\big) \notag \\ &\lesssim \sum_{j \in B_2} k(j+1) \exp\Big(-\frac{u(-j)-u(-(j+1))-u(1)}{2}\Big) \Upsilon\big(u(-j)-u(-(j+1))-u(1)\big) \notag \\ &\leq \sum_{j \in \mathbb{Z}}  k(j+1) e^{u(1)} \Upsilon\big(u(-j)-u(-(j+1))-u(1)\big) \lesssim \Psi_{2,\Upsilon}(u)(0).
\label{Case 2b}
\end{align} 
In the following we will consider several index sets $C_i \ (i=1,2,\dots,9)$ such that $\mathbb{N} = \bigcup\limits_{i}^9 C_i$. For any $i$, we will show that 
\begin{align*}
\sum_{j \in C_i} k(j) F_\gamma \big(w(j+1)-w(j)\big) \lesssim \Psi_{2,\Upsilon}(u)(0),
\end{align*}
from which we deduce that (\ref{Step 2 to show}) holds true.

\underline{\emph{Subcase 1. Indices $j$ with $j \in A \cap B$:}} In all calculations in this subcase we will use Lemma \ref{Lemma 1}. For $C_1 := A_1 \cap B_1$ we get
\begin{align*}
\sum_{j \in C_1} & k(j) F_\gamma\big(w(j+1)-w(j)\big) \lesssim \sum_{j \in C_1} k(j) F_\gamma \big(w(j+1)-w(j)-w(1)\big) + \sum_{j \in C_1} k(j) F_\gamma \big(w(1)\big) \\ &\lesssim \sum_{j \in A_1} k(j) F_\gamma \big(u(j+1)-u(j)-u(1)\big) + \sum_{j \in B_1} k(j) F_\gamma \big(u(-(j+1))-u(-j)-u(-1)\big) \\ & \quad\quad+ \sum_{j=1}^\infty k(j) F_\gamma \big(w(1)\big) \\ &\lesssim \Psi_{2,\Upsilon}(u)(0).
\end{align*}
For $C_2 := A_2 \cap B_2$ we argue analogeously by considering 
\begin{align*}
\sum_{j \in C_2} k(j) F_\gamma \big(w(j+1)-w(j)\big) &\lesssim \sum_{j \in C_2} k(j) F_\gamma \big(w(j+1)-w(j)+w(1)\big) + \sum_{j \in C_2} k(j) F_\gamma \big(w(1)\big) \\ &\lesssim \Psi_{2,\Upsilon}(u)(0).
\end{align*}
Defining $C_3 := A_1 \cap B_2$, we have
\begin{align*}
\sum_{j \in C_3} & k(j) F_\gamma \big(w(j+1)-w(j)\big) = \sum_{j \in C_3} k(j) F_\gamma \Big(w(j+1)-w(j)-\frac{u(1)}{2}+\frac{u(1)}{2}\Big) \\ &\lesssim \sum_{j \in A_1} k(j) F_\gamma \big(u(j+1)-u(j)-u(1)\big) + \sum_{j \in B_2} k(j) F_\gamma \big(u(-(j+1))-u(-j)+u(1)\big) \\ &\lesssim \Psi_{2,\Upsilon}(u)(0)
\end{align*}
and finally for $C_4 := A_2 \cap B_1$ we get
\begin{align*}
\sum_{j \in C_4}& k(j) F_\gamma \big(w(j+1)-w(j)\big) = \sum_{j \in C_4} k(j) F_\gamma \Big(w(j+1)-w(j)-\frac{u(-1)}{2}+\frac{u(-1)}{2}\Big) \\ &\lesssim \sum_{j \in A_2} k(j) F_\gamma \big(u(j+1)-u(j)+u(-1)\big) + \sum_{j \in B_1} k(j) F_\gamma \big(u(-(j+1))-u(-j)-u(-1)\big) \\ &\lesssim \Psi_{2,\Upsilon}(u)(0).
\end{align*} 
Note that $A \cap B = \bigcup\limits_{i=1}^4 C_i$.

\underline{\emph{Subcase 2. Indices $j$ with $j \in A^\mathrm{C} \cap B^\mathrm{C} =: C_5$:}} Since $j \in A^\mathrm{C}$ we have
\begin{equation}
u(-1) < u(j+1)-u(j) < -u(1).
\label{Proof:Subcase4.1}
\end{equation}
As we also have $j \in B^\mathrm{C}$ we further see that
\begin{equation}
u(1) < u(-(j+1))-u(-j) < -u(-1).
\label{Proof:Subcase3.1}
\end{equation}
Adding these two inequalities we arrive at 
\begin{equation*}
2w(1) < 2\big(w(j+1)-w(j)\big) < -2w(1),
\end{equation*}
i.e. $\vert w(j+1)-w(j) \vert < \vert w(1) \vert$. This yields
\begin{equation*}
\sum_{j \in C_5} k(j) F_\gamma \big(w(j+1)-w(j)\big) \leq \sum_{j=1}^\infty k(j) F_\gamma\big(w(1)\big) = \frac{1}{2} \vert k \vert _1 F_\gamma\big(w(1)\big) \lesssim \Psi_{2,\Upsilon}(u)(0).
\end{equation*}

\underline{\emph{Subcase 3. Indices $j$ with $j \in A \cap B^\mathrm{C}$:}} 
By $j \in B^\mathrm{C}$ we know that (\ref{Proof:Subcase3.1}) is satisfied. Hence there holds
\begin{align*}
2w(1) < u(-(j+1))-u(-j)+u(-1) < 0 \text{ and } 0 < u(-(j+1))-u(-j)-u(1) < -2w(1),
\end{align*}
which yields 
\begin{align}
\sum_{j \in B^\mathrm{C}} k(j) F_\gamma \big(u(-(j+1))-u(-j)+u(-1)\big) \leq \frac{1}{2} \vert k \vert_1 F_\gamma \big(2 w(1)\big) \lesssim \Psi_{2,\Upsilon}(u)(0)
\label{Case 4.2.1}
\end{align}
and
\begin{align}
\sum_{j \in B^\mathrm{C}} k(j) F_\gamma \big(u(-(j+1))-u(-j)-u(1) \big) \leq \frac{1}{2} \vert k \vert_1 F_\gamma \big(2 w(1)\big) \lesssim \Psi_{2,\Upsilon}(u)(0).
\label{Case 4.2.2}
\end{align}
Further, we infer from \eqref{Proof:Subcase3.1} that $w(1)<0$ holds.
\begin{itemize}
\item \dashuline{$j \in A_1 \cap A_2 \cap B^\mathrm{C}$:} Since $j \in A_1 \cap A_2$ we have
\begin{align*}
-u(1) \leq u(j+1)-u(j) \leq u(-1),
\end{align*} 
from which we would get $w(1) \geq 0$. As we have already seen that $w(1)<0$ must hold, this case cannot happen.
\item \dashuline{$j \in A_1 \cap A_2^\mathrm{C} \cap B^\mathrm{C} =: C_6$:} Since $j \in C_6$ we have
\begin{align*}
u(j+1)-u(j) \geq -u(1) \text{ and } u(j+1)-u(j)>u(-1).
\end{align*}
In particular, together with (\ref{Proof:Subcase3.1}) we observe that $2\big(w(j+1)-w(j)\big) \geq u(1)-u(1) =0$. We have to distinguish two different cases.
\begin{itemize}
\item[(i)] Looking first at the case in which $u(1) \leq 0$, we get with (\ref{Case 1a}), (\ref{Case 4.2.2}) and Lemma \ref{Lemma 1}
\begin{align*}
&\sum_{j \in C_6} k(j) F_\gamma \big(w(j+1)-w(j)\big) \leq \sum_{j \in C_6} k(j) F_\gamma \big(w(j+1)-w(j)-u(1)\big) \\ &\lesssim \sum_{j \in A_1} k(j) F_\gamma \big(u(j+1)-u(j)-u(1)\big) + \sum_{j \in B^\mathrm{C}} k(j) F_\gamma \big(u(-(j+1))-u(-j)-u(1)\big) \\ &\lesssim \Psi_{2,\Upsilon}(u)(0).
\end{align*}
\item[(ii)] Consider now the case $u(1)>0$. Then we know by (\ref{Proof:Subcase3.1}) that $u(-(j+1))-u(-j) > 0$ and $u(-1)<0$ hold. By Lemma \ref{Abschaetzung F Phi} we hence observe that
\begin{align*}
F_\gamma \big(u(-(j+1))-u(-j)\big) \lesssim e^{u(-1)} \Upsilon\big(u(-(j+1))-u(-j)-u(-1)\big), 
\end{align*}
and therefore (\ref{Estimate 2}) shows
\begin{align*}
\sum_{j \in C_6} k(j) F_\gamma\big(u(-(j+1))-u(-j)\big) &\lesssim \sum_{j=1}^\infty k(j) e^{u(-1)} \Upsilon\big(u(-(j+1))-u(-j)-u(-1)\big) \\ &\lesssim \Psi_{2,\Upsilon}(u)(0).
\end{align*}
As $j \in B_2^\mathrm{C}$, we know that $u(-(j+1))-u(-j)-u(1) > 0$. Thus, together with (\ref{Case 1a}) and Lemma \ref{Lemma 1}, we obtain that
\begin{align*}
& \sum_{j \in C_6} k(j) F_\gamma \big(w(j+1)-w(j)\big) \\ &\leq \sum_{j \in C_6} k(j) F_\gamma \Big(w(j+1)-w(j)+\frac{u(-(j+1))-u(-j)-u(1)}{2}\Big) \\ &\lesssim \sum_{j \in A_1} k(j) F_\gamma \big(u(j+1)-u(j)-u(1)\big) + \sum_{j \in C_6} k(j) F_\gamma \big(u(-(j+1))-u(-j)\big) \\ &\lesssim \Psi_{2,\Upsilon}(u)(0).
\end{align*}
\end{itemize}
\item \dashuline{$j \in A_1^\mathrm{C} \cap A_2 \cap B^\mathrm{C} =: C_7$:} Since $j \in C_7$ we get 
\begin{align*}
u(j+1)-u(j) < -u(1) \text{ and } u(j+1)-u(j) \leq u(-1).
\end{align*}
In particular, together with (\ref{Proof:Subcase3.1}), we observe that $2\big(w(j+1)-w(j)\big) \leq u(-1)-u(-1) =0$. We distinguish the following two cases.
\begin{itemize}
\item[(i)] Looking first at the case in which $u(-1) \leq 0$ we get with (\ref{Case 1b}), (\ref{Case 4.2.1}) and Lemma \ref{Lemma 1}
\begin{align*}
&\sum_{j \in C_7} k(j) F_\gamma \big(w(j+1)-w(j)\big) \leq \sum_{j \in C_7} k(j) F_\gamma \big(w(j+1)-w(j)+u(-1)\big) \\ &\lesssim \sum_{j \in A_2} k(j) F_\gamma \big(u(j+1)-u(j)+u(-1)\big) + \sum_{j \in B^\mathrm{C}} k(j) F_\gamma \big(u(-(j+1))-u(-j)+u(-1)\big) \\ &\lesssim \Psi_{2,\Upsilon}(u)(0).
\end{align*}
\item[(ii)] Consider now the case $u(-1)>0$. Then we know by (\ref{Proof:Subcase3.1}) that $u(-(j+1))-u(-j) < 0$ and $u(1)<0$. By Lemma \ref{Abschaetzung F Phi} we hence obtain that
\begin{align*}
F_\gamma \big(u(-j)-u(-(j+1))\big) \lesssim e^{u(1)} \Upsilon\big(u(-j)-u(-(j+1))-u(1)\big),
\end{align*}
and therefore (\ref{Estimate 4}), together with (\ref{Kern Bedingung Wachsum}), shows that
\begin{align*}
\sum_{j \in C_7} k(j) F_\gamma \big(u(-j)-u(-(j+1))\big) &\lesssim \sum_{j=1}^\infty k(j+1) e^{u(1)} \Upsilon\big(u(-j)-u(-(j+1))-u(1)\big) \\ &\lesssim \Psi_{2,\Upsilon}(u)(0).
\end{align*}
As $j \in B_1^\mathrm{C}$, we know that $u(-(j+1))-u(-j)+u(-1) < 0$. Thus, together with (\ref{Case 1b}) and Lemma \ref{Lemma 1}, we find
\begin{align*}
& \sum_{j \in C_7} k(j) F_\gamma \big(w(j+1)-w(j)\big) \\ &\leq \sum_{j \in C_7} k(j) F_\gamma \Big(w(j+1)-w(j)+\frac{u(-(j+1))-u(-j)+u(-1)}{2}\Big) \\ &\lesssim \sum_{j \in A_2} k(j) F_\gamma \big(u(j+1)-u(j)+u(-1)\big) + \sum_{j \in C_7}^\infty k(j) F_\gamma \big(u(-j)-u(-(j+1))\big) \\ &\lesssim \Psi_{2,\Upsilon}(u)(0).
\end{align*}
\end{itemize}
\end{itemize}

\underline{\emph{Subcase 4. Indices $j$ with $j \in A^\mathrm{C} \cap B$:}} By $j \in A^\mathrm{C}$ we know that (\ref{Proof:Subcase4.1}) is satisfied. Hence we observe that
\begin{align*}
2w(1) < u(j+1)-u(j)+u(1) < 0 \text{ and } 0 < u(j+1)-u(j)-u(-1) < -2w(1).
\end{align*}
This gives 
\begin{align}
\sum_{j \in A^\mathrm{C}} k(j) F_\gamma \big(u(j+1)-u(j)+u(1)\big) \leq \frac{1}{2} \vert k \vert_1 F_\gamma \big(2 w(1)\big) \lesssim \Psi_{2,\Upsilon}(u)(0)
\label{Case 4.1.1}
\end{align}
and 
\begin{align}
\sum_{j \in A^\mathrm{C}} k(j) F_\gamma \big(u(j+1)-u(j)-u(-1)\big) \leq \frac{1}{2} \vert k \vert_1 F_\gamma \big(2 w(1)\big) \lesssim \Psi_{2,\Upsilon}(u)(0).
\label{Case 4.1.2}
\end{align}
Moreover, we also deduce from \eqref{Proof:Subcase4.1} that $w(1)<0$ holds.
\begin{itemize}
\item \dashuline{$j \in A^\mathrm{C} \cap B_1 \cap B_2$:} Since $j \in B_1 \cap B_2$ we have
\begin{align*}
-u(-1) \leq u(-(j+1))-u(-j) \leq u(1),
\end{align*}
from which we would get $w(1) \geq 0$. As we have already seen that $w(1)<0$ must hold, this case cannot occur.
\item \dashuline{$j \in A^\mathrm{C} \cap B_1 \cap B_2^\mathrm{C} =: C_8$:} Since $j \in C_8$ we get
\begin{align*}
u(-(j+1))-u(-j) \geq -u(-1) \text{ and } u(-(j+1))-u(-j)>u(1).
\end{align*}
In particular, together with (\ref{Proof:Subcase4.1}), we observe that $2\big(w(j+1)-w(j)\big) \geq u(-1)-u(-1) =0$. We again distinguish two cases.
\begin{itemize}
\item[(i)] Looking first at the case in which $u(-1) \leq 0$ we get with (\ref{Case 2a}), (\ref{Case 4.1.2}) and Lemma \ref{Lemma 1}
\begin{align*}
& \sum_{j \in C_8} k(j) F_\gamma \big(w(j+1)-w(j)\big) \leq \sum_{j \in C_8} k(j) F_\gamma \big(w(j+1)-w(j)-u(-1)\big) \\ &\lesssim \sum_{j \in A^\mathrm{C}} k(j) F_\gamma \big(u(j+1)-u(j)-u(-1)\big) + \sum_{j \in B_1} k(j) F_\gamma \big(u(-(j+1))-u(-j)-u(-1)\big) \\ &\lesssim \Psi_{2,\Upsilon}(u)(0).
\end{align*}
\item[(ii)] Consider now the case $u(-1)>0$. Then we know from (\ref{Proof:Subcase4.1}) that $u(j+1)-u(j) > 0$ and $u(1)<0$. By Lemma \ref{Abschaetzung F Phi} we hence observe that
\begin{align*}
F_\gamma \big(u(j+1)-u(j)\big) \lesssim e^{u(1)} \Upsilon\big(u(j+1)-u(j)-u(1)\big), 
\end{align*}
and consequently (\ref{Estimate 1}) shows that
\begin{align*}
\sum_{j \in C_8} k(j) F_\gamma \big(u(j+1)-u(j)\big) \lesssim \sum_{j=1}^\infty k(j) e^{u(1)} \Upsilon\big(u(j+1)-u(j)-u(1)\big) \lesssim \Psi_{2,\Upsilon}(u)(0).
\end{align*}
As $j \in A_2^\mathrm{C}$, we know that $u(j+1)-u(j)-u(-1) > 0$. Thus, together with (\ref{Case 2a}) and Lemma \ref{Lemma 1}, we see that
\begin{align*}
\sum_{j \in C_8}&  k(j) F_\gamma \big(w(j+1)-w(j)\big) \leq \sum_{j \in C_8} k(j) F_\gamma \Big(w(j+1)-w(j)+\frac{u(j+1)-u(j)-u(-1)}{2}\Big) \\ &\lesssim \sum_{j \in B_1} k(j) F_\gamma \big(u(-(j+1))-u(-j)-u(-1)\big) + \sum_{j \in C_8} k(j) F_\gamma \big(u(j+1)-u(j)\big) \\ &\lesssim \Psi_{2,\Upsilon}(u)(0).
\end{align*}
\end{itemize}
\item \dashuline{$j \in A^\mathrm{C} \cap B_1^\mathrm{C} \cap B_2 =: C_9$:} Since $j \in C_9$ we get 
\begin{align*}
u(-(j+1))-u(-j) < -u(-1) \text{ and } u(-(j+1))-u(-j) \leq u(1).
\end{align*}
In particular, together with (\ref{Proof:Subcase4.1}), we observe that $2\big(w(j+1)-w(j)\big) \leq u(1)-u(1) =0$. We distinguish the following two cases.
\begin{itemize}
\item[(i)] Looking first at the case in which $u(1) \leq 0$ we get with (\ref{Case 2b}), (\ref{Case 4.1.1}) and Lemma \ref{Lemma 1}
\begin{align*}
& \sum_{j \in C_9} k(j) F_\gamma \big(w(j+1)-w(j)\big) \leq \sum_{j \in C_9} k(j) F_\gamma \big(w(j+1)-w(j)+u(1)\big) \\ &\lesssim \sum_{j \in A^\mathrm{C}} k(j) F_\gamma \big(u(j+1)-u(j)+u(1)\big) + \sum_{j \in B_2} k(j) F_\gamma \big(u(-(j+1))-u(-j)+u(1)\big) \\ &\lesssim \Psi_{2,\Upsilon}(u)(0).
\end{align*}
\item[(ii)] Consider now the case $u(1)>0$. Then we know by (\ref{Proof:Subcase4.1}) that $u(j+1)-u(j) < 0$ and $u(-1)<0$. By Lemma \ref{Abschaetzung F Phi} we hence observe that
\begin{align*}
F_\gamma \big(u(j)-u(j+1)\big) \lesssim e^{u(-1)} \Upsilon\big(u(j)-u(j+1)-u(-1)\big), 
\end{align*}
and thus (\ref{Estimate 3}), together with (\ref{Kern Bedingung Wachsum}), shows that
\begin{align*}
\sum_{j \in C_9} k(j) F_\gamma \big(u(j)-u(j+1)\big) &\lesssim \sum_{j=1}^\infty k(j+1) e^{u(-1)} \Upsilon\big(u(j)-u(j+1)-u(-1)\big) \\ &\lesssim \Psi_{2,\Upsilon}(u)(0).
\end{align*}
As $j \in A_1^\mathrm{C}$, we know that $u(j+1)-u(j)+u(1) < 0$. Thus, together with (\ref{Case 2b}) and Lemma \ref{Lemma 1}, we find
\begin{align*}
\sum_{j \in C_9}^\infty & k(j) F_\gamma \big(w(j+1)-w(j)\big) \leq \sum_{j \in C_9} k(j) F_\gamma \Big(w(j+1)-w(j)+\frac{u(j+1)-u(j)+u(1)}{2}\Big) \\ &\lesssim \sum_{j \in B_2} k(j) F_\gamma \big(u(-(j+1))-u(-j)+u(1)\big) + \sum_{j \in C_9} k(j) F_\gamma \big(u(j)-u(j+1)\big) \\ &\lesssim \Psi_{2,\Upsilon}(u)(0).
\end{align*}
\end{itemize}
\end{itemize}
As $\N = \bigcup_{i=1}^9 C_i$, we conclude that
\begin{equation*}
\sum_{j=1}^\infty k(j) F_\gamma \big(w(j+1)-w(j)\big) \leq \sum_{i=1}^9 \sum_{j \in C_i} k(j) F_\gamma \big(w(j+1)-w(j)\big)\lesssim \Psi_{2,\Upsilon} (u)(0),
\end{equation*} 
i.e. (\ref{Step 2 to show}) holds true.

\emph{Step 3: Estimating a weighted $l_\gamma$-norm of $w$ by $\Psi_{2,\Upsilon}(u)(0)$.} Let $N \in \mathbb{N}$ with $N \geq 2$. Then by Lemma \ref{Lemma 1} we have for any $\delta \in (0,\gamma-1)$
\begin{align*}
\sum_{j=1}^N \frac{k(j)}{j^{\gamma-1}} F_\gamma \big(w(j+1)\big) &= \sum_{j=1}^N \frac{k(j)}{j^{\gamma-1}} F_\gamma \big(w(j+1)-w(j)+w(j)\big) \\ &\leq \sum_{j=1}^N \frac{k(j)}{j^{\gamma-1}} \tilde{C} j^{\gamma-1} F_\gamma \big(w(j+1)-w(j)\big) +\sum_{j=1}^N \frac{k(j)}{j^{\gamma-1}} \Big(1+\frac{\delta}{j}\Big) F_\gamma \big(w(j)\big) \\ &\leq \tilde{C} \sum_{j=1}^N k(j) F_\gamma \big(w(j+1)-w(j)\big) + k(1) (1+\delta) F_\gamma \big(w(1)\big) \\ &\quad\quad + \sum_{j=1}^N \frac{k(j+1)}{(j+1)^{\gamma-1}} \Big(1+\frac{\delta}{j+1}\Big) F_\gamma \big(w(j+1)\big).
\end{align*}
Due to \emph{Step 2}, we observe that 
\begin{align*}
\sum_{j=1}^N \Lambda(j,\delta) F_\gamma \big(w(j+1)\big) \lesssim \Psi_{2,\Upsilon}(u)(0),
\end{align*}
where (using the monotonicity of $k$)
\begin{align*}
\Lambda (j,\delta) &= \frac{k(j)}{j^{\gamma-1}}-\frac{k(j+1)}{(j+1)^{\gamma-1}}-\delta \frac{k(j+1)}{(j+1)^\gamma} \\ &\geq k(j+1) \Big(\frac{1}{j^{\gamma-1}}-\frac{1}{(j+1)^{\gamma-1}}-\frac{\delta}{(j+1)^\gamma} \Big)\\ &= k(j+1)\Big(\frac{(j+1)^{\gamma-1} - j^{\gamma-1}}{j^{\gamma-1}(j+1)^{\gamma-1}} - \frac{\delta}{(j+1)^\gamma} \Big).
\end{align*}
By convexity of the mapping $r \mapsto r^{\gamma-1} \ (\gamma \geq 2)$, we have that $(j+1)^{\gamma-1}-j^{\gamma-1} \geq (\gamma-1) j^{\gamma-2}$ and hence
\begin{equation*}
\Lambda (j,\delta) \geq k(j+1) \Big(\frac{\gamma-1}{j (j+1)^{\gamma-1}} - \frac{\delta}{(j+1)^\gamma}\Big) \geq k(j+1) \frac{\gamma-1-\delta}{(j+1)^\gamma}.
\end{equation*}
Thus, 
\begin{equation}
\sum_{j=1}^N \frac{k(j+1)}{(j+1)^\gamma} F\big(w(j+1)\big) \leq \frac{1}{\gamma-1-\delta} \sum_{j=1}^N \Lambda (j,\delta) F\big(w(j+1)\big) \lesssim \Psi_{2,\Upsilon}(u)(0).
\label{Step 3 Summe bis N}
\end{equation}
Taking the limit as $N \to \infty$ in (\ref{Step 3 Summe bis N}) finally gives
\begin{equation*}
\sum_{j=1}^\infty \frac{k(j+1)}{(j+1)^\gamma} F\big(w(j+1)\big) \lesssim \Psi_{2,\Upsilon}(u)(0).
\end{equation*}

\emph{Step 4: Combining Step 1 and Step 3.} If we now combine \emph{Step 1} and \emph{Step 3}, we arrive at
\begin{align*}
F_\gamma \big(Lu(0)\big) &\leq 2^\gamma C_\alpha (k)^{\gamma-1} \sum_{j=1}^\infty \frac{k(j)}{j^\gamma} F_\gamma \big(w(j)\big) \\ &= 2^\gamma C_\alpha (k)^{\gamma-1} \Big(k(1)F_\gamma \big(w(1)\big) + \sum_{j=1}^\infty \frac{k(j+1)}{(j+1)^\gamma} F_\gamma \big(w(j+1)\big) \Big) \lesssim \Psi_{2,\Upsilon}(u)(0),
\end{align*}
i.e. there is a $d>0$ such that $\Psi_{2,\Upsilon}(u)(0) \geq \frac{1}{d} F_\gamma \big(Lu(0)\big) = \frac{1}{d} \big \vert Lu(0)\big \vert ^\gamma$. The corresponding estimate holds at any $x \in \mathbb{Z}$, by Lemma \ref{Lemma u(0)=0} and thus (\ref{Theorem Strong CD-inequ}) is proved. Considering now the $CD$-function $F = \frac{1}{d} F_\gamma \vert_{[0,\infty)}$, we finally conclude that the operator $L$ generated by $k$ satisfies $CD_\Upsilon(0,F)$.
\end{proof}

\theoremstyle{remark}
\begin{bemerkungen}
\label{BemerkungSatzAlphatesMom}
(i) The statement in Theorem \ref{Satz alphates Moment} is even stronger than a $CD_\Upsilon(0,F)$-condition as estimate (\ref{Theorem Strong CD-inequ}) also gives a non-trivial lower bound for $\Psi_{2,\Upsilon}(u)(x)$ in case that $-Lu(x)<0$. However, in the applications one only needs the estimate with $-Lu(x)\ge 0$.

(ii) Setting $\alpha = 2$ in estimate (\ref{Theorem Strong CD-inequ}) yields the validity of $CD_\Upsilon(0,d)$. More precisely, every operator $L$ generated by a kernel with finite second moment that is non-increasing on $\mathbb{N}$ and for which we find a suitable constant $C \geq 1$ such that (\ref{Kern Bedingung Wachsum}) holds satisfies (\ref{Theorem Strong CD-inequ}) with the $CD$-function $F(x) =\frac{1}{d} x^2$. As the condition $CD(0,d)$ is necessary for $CD_\Upsilon (0,d)$ to hold (cf. \cite[Remark 2.4(ii)]{WEB}), Theorem \ref{Satz alphates Moment} can be seen as a generalization of \cite[Theorem 4.1]{SWZ1} aside from the weak additional assumption \eqref{Kern Bedingung Wachsum}.

(iii) The condition \eqref{Kern Bedingung Wachsum} is counter-intuitive at a first glance, since it states that the speed of decrease is controlled. It is in particular violated for kernels that have finite support, which are usually considered as the best possible situations. However, there also exist kernels that violate condition \eqref{Kern Bedingung Wachsum} not for the reason of decaying too fast. Indeed, we will now present a class of integrable kernels that do neither satisfy \eqref{Kern Bedingung Wachsum} nor the mild integrability condition \eqref{conditiondelta}. Further, even the weaker condition \eqref{conditionlog} is violated for a certain range of parameters. Let the kernel $k$ be given as in Remark \ref{rem:kernelforcond2compare} for some $\alpha \geq 1$. Further, let $h(n)=n!$ for $n \in \mathbb{N}$. Then, we define the kernel $\tilde{k}$ by
\begin{equation*}
\tilde{k}(j)=k(h(n-1)), \text{ if } h(n-1) \leq |j| < h(n), n \geq 2.
\end{equation*}
We have that
\begin{align*}
\frac{\tilde{k}(h(n)-1)}{\tilde{k}(h(n))} = \frac{k(h(n-1))}{k(h(n))} &= \frac{n! \log(1+n!)^{1+\alpha}}{(n-1)!\log(1+(n-1)!)^{1+\alpha}}\\
&= n \Big(\frac{\log(1+n!)}{\log(1+(n-1)!)}\Big)^{1+\alpha}.
\end{align*}
By Stirling's formula, we know that $\log(n!)$ behaves for large $n$ like $n \log(n)$. Therefore, we conclude from the above calculation that \eqref{Kern Bedingung Wachsum} is not satisfied for the kernel $\tilde{k}$. Besides that, the kernel $\tilde{k}$ is integrable but the mild condition \eqref{conditiondelta} is not satisfied. Indeed,
for $N \geq 2$ and $\delta \in[0,1)$ we have that
\begin{align*}
\sum\limits_{j =1}^{h(N)-1} \tilde{k}(j)^{1-\delta} &= \sum\limits_{n =1}^N \big( h(n)-h(n-1))\, k(h(n-1))^{1-\delta} = \sum\limits_{n=1}^N (n-1)!\, (n-1)\, k((n-1)!)^{1-\delta}\\
&= \sum\limits_{n=1}^N \frac{((n-1)!)^\delta (n-1)}{\log(1+(n-1)!)^{(1+\alpha)(1-\delta)}}.
\end{align*} 
Employing again the asymptotic behaviour from Stirling's formula, we conclude that the denominator in the latter sum behaves for large $n$ like $(n-1)^{(1+\alpha)(1-\delta)} \log (n)^{(1+\alpha)(1-\delta)}$. Since $\alpha \geq 1$ this sum converges as $N \to \infty$ if $\delta=0$. In the case of $\delta \in (0,1)$ the sum apparently does not converge. We now show that even \eqref{conditionlog} is violated for $\alpha \in [1,2)$. One can employ similar arguments to observe that \eqref{conditionlog} holds true if $\alpha\ge 2$. We have
\begin{align*}
\sum\limits_{j=1}^{h(N)-1} \tilde{k}(j) \log\Big( 2 + \frac{1}{\tilde{k}(j)}\Big) &= \sum\limits_{n=1}^N (n-1)! (n-1)k(h(n-1)) \log\Big( 2 + \frac{1}{k(h(n-1))}\Big)\\
&\geq \sum\limits_{n=1}^N \frac{(n-1)\log\big( (n-1)! (\log(1+(n-1)!)^{1+\alpha}\big)}{(\log(1+(n-1)!))^{1+\alpha}}\\
&\geq \sum\limits_{n=1}^N \frac{(n-1)\log\big( (n-1)!\big)}{(\log(1+(n-1)!))^{1+\alpha}}.
\end{align*}
Employing again Stirling's formula, the summands behave for large $n$ like
\begin{align*}
\frac{1}{(n-1)^{\alpha-1} (\log(n-1))^\alpha}
\end{align*}
and hence the sum tends to $\infty$ as $N\to \infty$ if $\alpha < 2$. To conclude, the range of $\alpha \in [1,2)$ leads to examples which neither satisfy \eqref{Kern Bedingung Wachsum} nor fit into the setting of Section \ref{Chapter3}.
\end{bemerkungen}
As a corollary of its proof, we can generalize the statement of Theorem \ref{Satz alphates Moment} to kernels that are comparable with kernels that meet the assumptions of Theorem \ref{Satz alphates Moment}.
\begin{korollar}
\label{KorollarMonotonieKern}
Let $k, \alpha, \gamma$ and $F$ be as in Theorem \ref{Satz alphates Moment}. Consider another symmetric kernel $\tilde{k}: \mathbb{Z} \to [0,\infty)$ such that there exist constants $m_1, m_2 > 0$ with 
\begin{equation}
m_1 k(j) \leq \tilde{k} (j) \leq m_2 k(j)  
\label{BedingungKorollarMonotonieKern}
\end{equation}
for any $j \in \N$. Let $C_\alpha(k)$ be given by \eqref{Kern Bedingung Moment}  and define
\begin{equation*}
C_\alpha(\tilde{k})= \sum\limits_{j \in \N} \tilde{k}(j)j^\alpha.
\end{equation*}
Then the operator $\tilde{L}$ generated by $\tilde{k}$ satisfies $CD_\Upsilon (0,\tilde{F})$ with the $CD$-function $\tilde{F} = \frac{1}{\tilde{d}} F$, where $\tilde{d} = \frac{C_\alpha(\tilde{k})^{\gamma-1}m_2}{C_\alpha(k)^{\gamma-1}m_1^2}$.
\end{korollar}

\begin{proof}
First, note that \eqref{BedingungKorollarMonotonieKern} implies that $C_\alpha(\tilde{k})<\infty$. 
Recalling the strategy of the proof of Theorem \ref{Satz alphates Moment}, we argue as in {\em Step 1} to obtain
\begin{align*}
F_\gamma\big(\tilde{L}u(0)\big) \leq 2^\gamma C_\alpha(\tilde{k})^{\gamma-1} \sum\limits_{j=1}^\infty \frac{\tilde{k}(j)}{j^\gamma}F_\gamma\big(w(j)\big) \leq 2^\gamma C_\alpha(\tilde{k})^{\gamma-1} m_2 \sum\limits_{j=1}^\infty \frac{k(j)}{j^\gamma}F_\gamma\big(w(j)\big),
\end{align*}
where $u \in \ell^\infty(\Z)$ and $w$ denotes the corresponding symmetrization.
Now, proceeding as in the proof of Theorem \ref{Satz alphates Moment}, we know that 
\begin{equation*}
2^\gamma \sum\limits_{j=1}^\infty \frac{k(j)}{j^\gamma}F_\gamma\big(w(j)\big) \leq \frac{d}{C_\alpha(k)^{\gamma-1}}\, \Psi_{2,\Upsilon}(u)(0),
\end{equation*}
where $d$ is as in Theorem \ref{Satz alphates Moment}.
Clearly, \eqref{BedingungKorollarMonotonieKern} yields $\Psi_{2,\Upsilon}(u)(0) \leq \frac{1}{m_1^2}\tilde{\Psi}_{2,\Upsilon}(u)(0)$, where $\tilde{\Psi}_{2,\Upsilon}$ denotes the corresponding operator for $\tilde{L}$. Hence, we end up with
\begin{align*}
F_\gamma\big(\tilde{L}u(0)\big) \leq \frac{C_\alpha(\tilde{k})^{\gamma-1} m_2 d}{C_\alpha(k)^{\gamma-1}m_1^2}\, \tilde{\Psi}_{2,\Upsilon}(u)(0)
\end{align*}
and the claim follows.
\end{proof}

Due to its particular importance, we provide a formulation of Theorem \ref{Satz alphates Moment} for the specific example of the operator $L_\beta$. Other important operators to which Theorem \ref{Satz alphates Moment} applies, too, include, for instance, 
those whose associated kernels are non-increasing and decay exponentially.
\begin{korollar}
\label{Korollar Potenzkern}
Let $\beta \in (1,\infty)$, $\alpha \in (1,2]$ such that $\alpha<\beta$ and set $\gamma = \frac{\alpha}{\alpha-1}$. Then there exists a constant $d>0$ such that the operator $L_\beta$ satisfies $CD_\Upsilon (0,F)$ with the $CD$-function $F = \frac{1}{d} F_\gamma \vert_{[0,\infty)}$.
\end{korollar}
\begin{proof}
Clearly, the kernel $k_\beta$ is non-increasing on $\mathbb{N}$ and has finite $\alpha$-th moment as $\alpha<\beta$. Furthermore, we have $k_\beta (j)\le 2^{1+\beta} k_\beta (j+1)$ for every $j\in \iN$, i.e.\ condition \eqref{Kern Bedingung Wachsum} is satisfied.
The claim now follows from Theorem \ref{Satz alphates Moment}.
\end{proof}

Regarding the operator $L_\beta$ with $\beta >1$, Corollary \ref{Korollar Potenzkern} yields an alternative $CD$-function to the one we obtained in Corollary \ref{powertypeCDfrombasicestimate}. Hence, it is natural to ask if this $CD$-function yields an improvement compared to the result of Corollary \ref{powertypeCDfrombasicestimate}. As the $CD$-function from Corollary \ref{powertypeCDfrombasicestimate} behaves better for large arguments, we only need to focus on the behaviour for small arguments. If $\beta>2$, we can choose $\alpha=2$ in Corollary \ref{Korollar Potenzkern}, which yields a quadratic behaviour of the corresponding $CD$-function near zero. This is best possible and, besides that, unreachable for the $CD$-function from Corollary \ref{powertypeCDfrombasicestimate} since $\frac{1+2\beta}{\beta}>2$ for any $\beta >0$. In order to compare our two candidates for $\beta \in (1,2]$, we need to determine those  $\beta>1$ for which one has
\begin{align}
\frac{1+2\beta}{\beta} < \frac{\beta}{\beta-1}.
\label{NotwBedingungBeta}
\end{align}
This inequality is equivalent to $\beta^2-\beta-1 <0$, which yields that \eqref{NotwBedingungBeta} holds if and only if $\beta \in \big(1,\frac{1+\sqrt{5}}{2}\big]$. In this case the $CD$-function of Corollary \ref{powertypeCDfrombasicestimate} behaves better for small arguments. If instead $\beta>\frac{1+\sqrt{5}}{2}$, the $CD$-function from Corollary \ref{KorollarMonotonieKern} dominates for small arguments. We summarize our findings in the following result.

\begin{satz}
\label{SatzCDFunctionPotenzkern}
Let $\beta \in (0,2)$ and set $\beta_* = \frac{1+\sqrt{5}}{2}$. Then there exist constants $c_i >0 \ (i=1,\dots,4)$ and some $\nu >0$ such that $L_\beta$ satisfies $CD_\Upsilon (0,F)$ with a $CD$-function $F$ that behaves like
\begin{equation}  \label{asympCDFunctionLessGolden}
F(x) \sim c_1 x^\frac{1+2\beta}{\beta} \text{ as } x \to 0 \,\,  
 \mbox{and}\,\, F(x)\sim c_2 e^{\nu x} \text{ as } x\to \infty
\end{equation}
if $\beta \in (0,\beta_*]$ and like 
\begin{equation}  \label{asympCDFunctionGrGolden}
F(x) \sim c_3 x^\frac{\beta-\varepsilon}{\beta-\varepsilon-1} \, \mbox{as}\;x\to 0 \,\,  
 \mbox{and}\,\, F(x)\sim c_4 e^{\nu x}\, \mbox{as}\;x\to \infty,
\end{equation}
if $\beta \in (\beta_*,2)$. Here, $\varepsilon \in (0,\beta-1)$ can be chosen arbitrarily small.
\end{satz}
Combining Theorem \ref{SatzCDFunctionPotenzkern} with Corollary \ref{KorollarMonotonieKern} and
having in mind that \eqref{Zusammenhang Kerne} holds, we deduce a corresponding result for fractional powers of the discrete Laplacian in one dimension.

\begin{korollar}
\label{KorollarCDFunktionFrakLaplace}
Let $\beta \in (0,2)$.
There exist constants $\tilde{c}_i >0 \ (i=1,\dots,4)$ and some $\nu >0$ such that 
the fractional discrete Laplacian in one dimension $-\big(-\Delta\big)^\frac{\beta}{2}$ satisfies $CD_\Upsilon(0,F)$ with a $CD$-function $F$ that behaves like
\begin{equation*}
F(x) \sim \tilde{c}_1 x^\frac{1+2\beta}{\beta} \, \mbox{as}\;x\to 0 \,\,  
 \mbox{and}\,\, F(x)\sim \tilde{c}_2  e^{\nu x}\, \mbox{as}\;x\to \infty
\end{equation*}
if $\beta \in (0,\beta_*]$ and like
\begin{equation*} 
F(x) \sim \tilde{c}_3 x^\frac{\beta-\varepsilon}{\beta-\varepsilon-1} \, \mbox{as}\;x\to 0 \,\,  
 \mbox{and}\,\, F(x)\sim \tilde{c}_4 e^{\nu x}\, \mbox{as}\;x\to \infty
\end{equation*}
if $\beta \in (\beta_*,2)$. Here, $\varepsilon \in (0,\beta-1)$ can again be chosen arbitrarily small.
\end{korollar}

\theoremstyle{remark}
\begin{bemerkung}
\label{BemerkungFrakLaplace}
It is also possible to directly invoke Theorem \ref{Satz alphates Moment} in order to obtain the statement of Corollary \ref{KorollarMonotonieKern} in the regime where $\beta\in (\beta_*,2)$. While the existence of the corresponding $\alpha$-th moment can be deduced from the similarity to the respective power-type kernel, one can verify \eqref{Kern Bedingung Wachsum} by a simple application of the fundamental identity of the Gamma function using the representation formula \eqref{FrakLaplacekernel}. The fact that $k_\beta^*$ is non-increasing on $\N$ follows from a straightforward calculation, see Example \ref{HarnackExamples}(ii). 
\end{bemerkung}


\section{Li-Yau inequality}\label{sec:LiYau}
The aim of this section is to prove a Li-Yau type inequality for an operator $L$ of the form (\ref{eq:operator}) satisfying $CD_\Upsilon (0,F)$ with a certain $CD$-function $F$. For functions $u$ satisfying this Li-Yau estimate we will also be able to prove a Harnack inequality in Section \ref{Chapter6}.

Again we will have a special focus on the operators $L_\beta$ and the fractional discrete Laplacian in one dimension $-(-\Delta)^\frac{\beta}{2}$ with $\beta \in (0,2)$, for which we already know from the previous section that a certain $CD_\Upsilon (0,F)$ condition is satisfied. As the power-type kernel has unbounded support, we will need a suitable weight function which we define in the following lemma. Note that for any $\alpha \in (0,1]$ there holds 
\begin{equation}
\big \vert \vert x \vert^\alpha -\vert y \vert ^\alpha \big \vert \leq \vert x-y \vert^\alpha, \ x,y \in \mathbb{R}.
\label{UnglAlpha}
\end{equation}
This follows from the triangle inequality and the inequality $(a+b)^\alpha \leq a^\alpha + b^\alpha, \ a,b\ge 0, \alpha \in (0,1]$.

\begin{lemma}
\label{Lemma2LiYau}
For $\alpha \in (0,1]$ and $R \in \mathbb{N}$ let the function $\psi_R: \mathbb{Z} \to (0,1]$ be defined by
\begin{align*}
\psi_R(x) := \begin{cases}
        \begin{aligned}
          & \Big(\frac{R}{\vert x \vert}\Big)^\alpha, \text{ if } \vert x \vert \geq R, \\
          & 1, \text{ if } \vert x \vert < R.\\
        \end{aligned}
\end{cases}
\end{align*}
Then there holds
\begin{equation*}
\frac{\vert \psi_R (x) - \psi_R(y) \vert}{\psi_R (x)} \leq \frac{1}{R^\alpha} \vert x-y \vert^\alpha, \ x,y \in \mathbb{Z}.
\end{equation*}
\end{lemma}

\begin{proof}
We distinguish several cases in which we will make use of (\ref{UnglAlpha}). 
\begin{itemize}
\item \emph{Case 1. $\vert x \vert, \vert y \vert \geq R$:} Here we may estimate as follows.
\begin{align*}
\frac{\vert \psi_R (x) - \psi_R(y) \vert}{\psi_R (x)} = \frac{\big \vert \vert y \vert^\alpha - \vert x \vert^\alpha \big \vert}{\vert y \vert^\alpha} \leq \frac{1}{R^\alpha} \big \vert \vert y \vert^\alpha - \vert x \vert^\alpha \big \vert \leq \frac{1}{R^\alpha} \vert x-y \vert^\alpha.
\end{align*}
\item \emph{Case 2. $\vert x \vert, \vert y \vert < R$:} Clear as the left hand side vanishes.
\item \emph{Case 3. $\vert x \vert \geq R, \vert y \vert < R$:} In this case there holds $\vert x \vert^\alpha - R^\alpha \leq \vert x \vert^\alpha - \vert y \vert^\alpha$ from which we deduce
\begin{align*}
\frac{\vert \psi_R (x) - \psi_R(y) \vert}{\psi_R (x)} = \frac{1}{R^\alpha} (\vert x \vert^\alpha -R^\alpha) \leq \frac{1}{R^\alpha} \big \vert \vert x \vert^\alpha - \vert y \vert^\alpha \big \vert \leq \frac{1}{R^\alpha}  \vert x-y \vert^\alpha.
\end{align*}
\item \emph{Case 4. $\vert x \vert < R, \vert y \vert \geq R$:} Analogeously to Case 3, there holds $\vert y \vert^\alpha - R^\alpha \leq \vert y \vert^\alpha - \vert x \vert^\alpha$ in this case, which gives
\begin{align*}
\frac{\vert \psi_R (x) - \psi_R(y) \vert}{\psi_R (x)} \leq \frac{1}{R^\alpha} (\vert y \vert^\alpha -R^\alpha) \leq \frac{1}{R^\alpha} \big \vert \vert x \vert^\alpha - \vert y \vert^\alpha \big \vert \leq \frac{1}{R^\alpha}  \vert x-y \vert^\alpha.
\end{align*}
\end{itemize}
\end{proof}

In the following we consider $CD$-functions $F$ such that the function $g:[0,\infty) \to [0,\infty)$ defined by
\begin{align}
g(x) = \begin{cases}
\frac{F(x)}{x}, \quad x > 0,\\
0, \quad x =0,
\end{cases}
\label{DefFunktiong}
\end{align}
satisfies the estimate
\begin{align}
g(x) +g(y) \leq g(x+ y), \quad x,y \in [0,\infty).
\label{BedinungSA}
\end{align}
Note, that this includes all $CD$-functions $F$ such that $x \mapsto x^{-2}F(x)$  is monotonically increasing on $(0,\infty)$.  Indeed, if $F$ enjoys the latter property
we may estimate as
\begin{align*}
g(x) + g(y) = x \frac{g(x)}{x} + y \frac{g(y)}{y} \leq (x+y) \frac{g(x+y)}{(x+y)} = g(x+y),\quad x,y> 0.
\end{align*}
Recall that all $CD$-functions of our interest can be bounded from below by a $CD$-function $\hat{F}$ of the form \eqref{eq:Fhat}, for which the choice of $M\geq \frac{2}{\delta}$ implies that $x \mapsto  x^{-2}\hat{F}(x)$ is monotonically increasing on $(0,\infty)$.
%


%

The main theorem of this section reads as follows.


\begin{satz}
\label{SatzLiYau}
Let $\alpha >0$ and $k$ be a kernel with finite $\alpha$-th moment. Let $L$ be the operator generated by $k$. Suppose that $L$ satisfies $CD_\Upsilon (0,F)$ with a $CD$-function $F$ such that the function $g$ from (\ref{DefFunktiong}) satisfies (\ref{BedinungSA}). Let $\varphi$ be the relaxation function associated with the function $2F$. Consider a bounded function $u: [0,\infty) \times \mathbb{Z} \to (0,\infty)$ that is $C^1$ in time and solves the equation 
\[
\partial_t u (t,x) - L u(t,x) = 0\quad\mbox{on}\; (0,\infty) \times \mathbb{Z}.
\]
Then for $v (t,x) = \log u(t,x)$ there holds $v(t,x+\cdot) \in \ell_k^1(\mathbb{Z})$ for any $(t,x) \in (0,\infty) \times \mathbb{Z}$ and
\begin{equation}
-L v (t,x) \leq \varphi (t), \ (t,x) \in (0,\infty) \times \mathbb{Z}
\label{LiYauInequ1}
\end{equation}
and thus 
\begin{equation}
\partial_t v (t,x) \geq \Psi_\Upsilon (v)(t,x) - \varphi (t), \ (t,x) \in (0,\infty) \times \mathbb{Z}.
\label{LiYauInequ2}
\end{equation}
\end{satz}

\begin{proof}
We first consider the bounded function $v_\varepsilon (t,x) := \log \big(u(t,x)+\varepsilon \big), \ \varepsilon >0,$ and show the estimate
\begin{equation}
-L v_\varepsilon (t,x) \leq \varphi (t), \ (t,x) \in (0,\infty) \times \mathbb{Z}, \ \varepsilon>0.
\label{LiYauInequEps1}
\end{equation}
Afterwards, we will investigate the limit as $\varepsilon\to 0$ in order to establish the general statement.

Let $\varepsilon>0$ and $R \in \mathbb{N}$ be arbitrarily fixed and consider the corresponding weight function $\psi_R (x)$ from Lemma \ref{Lemma2LiYau}. On $[0,\infty) \times \mathbb{Z}$ we define the function $G_\varepsilon$ by setting
\begin{equation*}
G_\varepsilon (t,x) = -\frac{L v_\varepsilon (t,x)}{\varphi (t)} \psi_R (x), \ t>0, x \in \mathbb{Z},
\end{equation*}
and $G_\varepsilon (0,x) = 0, \ x \in \mathbb{Z}$. Observe that $G_\varepsilon$ is continuous in time since $\varphi (t) \to \infty$ as $t \to 0+$ and $Lv_\varepsilon (t,x)$ is bounded. Our aim is to show that for any $(t,x) \in (0,\infty) \times \Z$ there holds 
\begin{align}
G_\varepsilon (t,x) \leq 1+ \frac{1}{\varphi (t)} h(R)
\label{ProofLiYauToShow}
\end{align}
with some function $h$ such that $h(R) \to 0$ as $R \to \infty$. Sending $R \to \infty$ then yields (\ref{LiYauInequEps1}).

We start by generalizing the two operators $\Psi_\Upsilon$ and $B_{\Upsilon'}$ introduced in Section 1.2. For a continuous function $H:\,\iR \rightarrow \iR$ and 
$w,\tilde{w} \in \ell^\infty(\mathbb{Z})$ we define (cf.\ \cite[Section 2]{WZ})
\begin{equation*}
\Psi_H(w)(x) = \sum\limits_{y \in \mathbb{Z}} k(x-y) H\big(w(y)-w(x)\big),\quad x\in \iZ,
\end{equation*}
and 
\begin{equation*}
B_{H}(w,\tilde{w})(x)= \sum\limits_{y \in \mathbb{Z}} k(x-y) H\big(w(y)-w(x)\big) 
\big( \tilde{w}(y)-\tilde{w}(x)\big),\quad x\in \iZ.
\end{equation*}

By \cite[Lemma 2.2]{WZ}, the evolution equation for $v_\varepsilon$ reads as
\begin{equation}
\partial_t v_\varepsilon - L v_\varepsilon = \Psi_\Upsilon(v_\varepsilon), \ (t,x) \in (0,\infty) \times \mathbb{Z}.
\label{GleichungLPsi}
\end{equation}
Applying $L$ to equation (\ref{GleichungLPsi}), we deduce that
\begin{equation}
\partial_t L v_\varepsilon = L \Psi_{\Upsilon^\prime} (v_\varepsilon), \ (t,x) \in (0,\infty) \times \mathbb{Z},
\label{GleichungLPsiPalmeStrich}
\end{equation}
due to the dominated convergence theorem. Observe that (see also \cite[Remark 2.9 (iii)]{WZ})
\begin{align*}
L \Psi_{\Upsilon^\prime} (v_\varepsilon) = 2 \Psi_{2,\Upsilon} (v_\varepsilon) + B_{\exp} (v_\varepsilon,L v_\varepsilon), \ (t,x) \in (0,\infty) \times \mathbb{Z}.
\end{align*}
Inserting this into (\ref{GleichungLPsiPalmeStrich}), we obtain
\begin{align*}
\partial_t L v_\varepsilon = 2 \Psi_{2,\Upsilon} (v_\varepsilon) + B_{\exp} (v_\varepsilon,L v_\varepsilon), \ (t,x) \in (0,\infty) \times \mathbb{Z},
\end{align*}
which in turn gives
\begin{align*}
\partial_t G_\varepsilon &= -\frac{\psi_R (x)}{\varphi (t)} \big(2 \Psi_{2,\Upsilon} (v_\varepsilon) + B_{\exp} (v_\varepsilon,L v_\varepsilon)\big) - \frac{\dot{\varphi} (t)}{\varphi (t)} G_\varepsilon. 
\end{align*}

Now let $t_1>0$ be arbitrarily fixed and suppose that $G_\varepsilon$ (restricted to the set $[0,t_1] \times \mathbb{Z}$) assumes its global maximum at $(t_*,x_*) \in [0,t_1] \times \mathbb{Z}$. We can assume w.l.o.g. that $G_\varepsilon (t_*,x_*) > 0$ as otherwise inequality (\ref{LiYauInequEps1}) follows by positivity of $\varphi$. Then by definition of $G_\varepsilon$ it is clear that $t_* >0$ and thus ($\partial_t G_\varepsilon)(t_*,x_*) \geq 0$. Thus, at the maximum point $(t_*,x_*)$ we get the estimate
\begin{align}
0 \leq -2 \Psi_{2,\Upsilon} (v_\varepsilon)(t_*,x_*) - B_{\exp} (v_\varepsilon,L v_\varepsilon)(t_*,x_*) - \frac{\dot{\varphi} (t_*)}{\psi_R (x_*)} G_\varepsilon (t_*,x_*).
\label{BeweisLiYau1}
\end{align}
We now consider the second term on the right-hand side of this inequality. As $u$ is bounded, by assumption, there exists an $M>0$ such that $u(j) \leq M, \ j \in \mathbb{Z}$. Using this, the maximum property of $G_\varepsilon$ at the point $(t_*,x_*)$ and Lemma \ref{Lemma2LiYau}, we find
\begin{align*}
-B_{\exp} &(v_\varepsilon,L v_\varepsilon)(t_*,x_*) = -\sum_{j \in \mathbb{Z}} k(j) e^{v_\varepsilon (t_*,x_*+j)-v_\varepsilon (t_*,x_*)}\big(L v_\varepsilon (t_*,x_*+j)-L v_\varepsilon (t_*,x_*)\big) \\ &= \sum_{j \in \mathbb{Z}} k(j) e^{v_\varepsilon (t_*,x_*+j)-v_\varepsilon (t_*,x_*)}\Big(\frac{\varphi (t_*)}{\psi_R (x_*)} \big(G_\varepsilon (t_*,x_*+j)-G_\varepsilon (t_*,x_*)\big) \\ & \hspace{6cm}+\frac{\psi_R (x_*)-\psi_R (x_*+j)}{\psi_R (x_*)} \big(-L v_\varepsilon (t_*,x_*+j)\big)\Big) \\ &\leq \sum_{j \in \mathbb{Z}} k(j) e^{v_\varepsilon (t_*,x_*+j)-v_\varepsilon (t_*,x_*)}\frac{\vert \psi_R (x_*)-\psi_R (x_*+j) \vert}{\psi_R (x_*)} \vert L v_\varepsilon (t_*,x_*+j) \vert \\ & \leq \frac{M+\varepsilon}{\varepsilon} \log \Big(\frac{M+\varepsilon}{\varepsilon}\Big) \vert k \vert_1 \frac{1}{R^\alpha} \sum_{j \in \mathbb{Z}} k(j) \vert j \vert^\alpha =: \frac{C(\varepsilon,M)}{R^\alpha}.
\end{align*}
Note that the last sum is finite since $k$ has finite $\alpha$-th moment. Using this estimate in (\ref{BeweisLiYau1}) yields
\begin{align*}
2 \Psi_{2,\Upsilon} (v_\varepsilon)(t_*,x_*) \leq - \frac{\dot{\varphi} (t_*)(-L v_\varepsilon (t_*,x_*))}{\varphi (t_*)}+ \frac{C(\varepsilon,M)}{R^\alpha}.
\end{align*}
We now apply the condition $CD_\Upsilon(0,F)$ and use the ODE for the relaxation function $\varphi$ to get
\begin{align*}
2 F\big(-L v_\varepsilon (t_*,x_*)\big) \leq \frac{2 F\big(\varphi (t_*)\big)}{\varphi (t_*)} \big(-L v_\varepsilon (t_*,x_*)\big)+ \frac{C(\varepsilon,M)}{R^\alpha}.
\end{align*}
The function $g$ from (\ref{DefFunktiong}) is invertible and therefore the previous inequality is equivalent to
\begin{align*}
g\big(-L v_\varepsilon (t_*,x_*)\big) \leq g\big(\varphi (t_*)\big)+ g\Big(g^{-1}\Big(\frac{C(\varepsilon,M)}{2 R^\alpha(-L v_\varepsilon (t_*,x_*))}\Big)\Big).
\end{align*}
Applying (\ref{BedinungSA}) now gives 
\begin{align*}
g\big(-L v_\varepsilon (t_*,x_*)\big) \leq g\Big(\varphi (t_*) + g^{-1}\Big(\frac{C(\varepsilon,M)}{2 R^\alpha (-L v_\varepsilon (t_*,x_*))}\Big)\Big).
\end{align*}
Since $F$ is a $CD$-function, $g$ is strictly increasing. We therefore find
\begin{align*}
-L v_\varepsilon (t_*,x_*) \leq \varphi (t_*) + g^{-1}\Big(\frac{C(\varepsilon,M)}{2 R^\alpha (-L v_\varepsilon (t_*,x_*))}\Big),
\end{align*}
which yields
\begin{align}
G_\varepsilon (t_*,x_*) \leq \frac{-L v_\varepsilon (t_*,x_*)}{\varphi (t_*)} &\leq 1+ \frac{1}{\varphi (t_*)} g^{-1}\Big(\frac{C(\varepsilon,M)}{2 R^\alpha (-L v_\varepsilon (t_*,x_*))}\Big) \notag \\ &\leq 1+ \frac{1}{\varphi (t_1)} g^{-1}\Big(\frac{C(\varepsilon,M)}{2 R^\alpha (-L v_\varepsilon (t_*,x_*))}\Big),
\label{LiYauBeweisFall2}
\end{align}
as $\varphi$ is strictly decreasing. We now distinguish two cases.
\begin{itemize}
\item \emph{Case 1. $-L v_\varepsilon (t_*,x_*) \leq R^{-\frac{\alpha}{2}}$:} In this case, the desired inequality (\ref{ProofLiYauToShow}) follows without the preceding calculations. As $(t_*,x_*)$ is a global maximum point of $G_\varepsilon$ restricted to the set $[0,t_1] \times \mathbb{Z}$, we have
\begin{align*}
G_\varepsilon(t_1,x) \leq G_\varepsilon (t_*,x_*) \leq \frac{\psi_R (x_*) }{\varphi (t_*) R^\frac{\alpha}{2}} \leq 1 + \frac{1}{\varphi (t_1) R^\frac{\alpha}{2}}, \quad x \in  \mathbb{Z}.
\end{align*}
\item \emph{Case 2. $-L v_\varepsilon (t_*,x_*) > R^{-\frac{\alpha}{2}}$:} We use again the maximum property of $G_\varepsilon$ restricted to the set $[0,t_1] \times \mathbb{Z}$ at $(t_*,x_*)$. With (\ref{LiYauBeweisFall2}) we now obtain
\begin{align*}
G_\varepsilon (t_1,x) \leq G_\varepsilon (t_*,x_*) \leq 1+ \frac{1}{\varphi (t_1)}  g^{-1}\Big(\frac{C(\varepsilon,M)}{2 R^\frac{\alpha}{2}}\Big), \quad x \in  \mathbb{Z},
\end{align*}
as $g^{-1}$ is strictly increasing.
\end{itemize}
Summarizing these results, we get for any $(t_1,x) \in (0,\infty) \times \mathbb{Z}$
\begin{align*}
\frac{(-L v_\varepsilon(t_1,x))\psi_R(x)}{\varphi (t_1)} = G_\varepsilon(t_1,x) \leq 1+\frac{1}{\varphi(t_1)} \max \Big\{\frac{1}{R^\frac{\alpha}{2}}, g^{-1} \Big(\frac{C(\varepsilon,M)}{2 R^\frac{\alpha}{2}}\Big)\Big\},
\end{align*}
which shows (\ref{ProofLiYauToShow}). Finally, sending $R \to \infty$ we arrive at
\begin{align*}
-L v_\varepsilon(t,x) \leq \varphi (t), \ (t,x) \in (0,\infty) \times \mathbb{Z}, \ \varepsilon>0.
\end{align*}
This shows (\ref{LiYauInequEps1}).

It remains to study the limit as $\varepsilon\to 0$. For arbitrarily fixed 
$(t,x) \in (0,\infty) \times \mathbb{Z}$, we consider the sets $S_1=\{j \in \mathbb{Z}: u(t,x+j)<\frac{1}{2}\}$ and $S_2= \mathbb{Z}\setminus S_1$. Let $\varepsilon \in (0,\frac{1}{2})$. From \eqref{LiYauInequEps1} we deduce
\begin{equation}\label{eq:rearrangedforBeppo}
0 \leq -\sum\limits_{j \in S_1} k(j) v_\varepsilon(t,x+j) \leq \varphi(t) + \sum\limits_{j \in S_2} k(j)\big( v_\varepsilon(t,x+j)-v_\varepsilon(t,x)\big) - v_\varepsilon(t,x) \sum\limits_{j \in S_1} k(j),
\end{equation}
where the lower bound follows from the definition of the set $S_1$. By the theorem of dominated convergence we observe that 
\begin{equation*}
\lim\limits_{\varepsilon \to 0} \sum\limits_{j \in S_2} k(j)\big( v_\varepsilon(t,x+j)-v_\varepsilon(t,x)\big)  = \sum\limits_{j \in S_2} k(j)\big( v(t,x+j)-v(t,x)\big).
\end{equation*}
Further, the monotone convergence theorem implies that
\begin{equation*}
\lim\limits_{\varepsilon \to 0} \Big(-\sum\limits_{j \in S_1} k(j) v_\varepsilon(t,x+j)\Big) = -\sum\limits_{j \in S_1} k(j) v(t,x+j).
\end{equation*}
Consequently, sending $\varepsilon\to 0$ in \eqref{eq:rearrangedforBeppo} yields that $v(t,x+\cdot) \in \ell_k^1(\mathbb{Z})$ and also, after a small rearrangement, that 
\begin{equation} \label{finalest2}
- L v (t,x) \leq \varphi(t).
\end{equation}
Finally, \eqref{LiYauInequ2} follows from \eqref{finalest2} by \cite[Lemma 2.2]{WZ}.
\end{proof}

%

\begin{bemerkung}
In the assumptions of Theorem \ref{SatzLiYau}, instead of (\ref{BedinungSA}) we could have also assumed that $g$ satisfies the estimate
\begin{equation*}
g(x)+g(y) \leq g(x+ \sigma y), x,y \in [0,\infty),
\end{equation*}
with some $\sigma \geq 1$. As all $CD$-functions of our interest satisfy this estimate with $\sigma =1$, we restricted ourselves to this case.
\end{bemerkung}

\section{Harnack inequality and heat kernel estimate}
\label{Chapter6}

One of the most important applications of the Li-Yau estimate lies in the derivation of a Harnack inequality.

\begin{satz}
\label{SatzHarnackUngl}
Let $L$ be an operator of the type (\ref{eq:operator}) and $k$ be the kernel corresponding to $L$. Suppose that $u:[0,\infty) \times \mathbb{Z} \to (0,\infty)$ is $C^1$ in time and such that $\Upsilon \big(\log u(t,x+\cdot)-\log u(t,x) \big)$ belongs to $\ell^1_k (\mathbb{Z})$ for all $t>0$ and $x \in \Z$. Suppose further that $u$ satisfies the differential Harnack estimate 
\begin{align}
\partial_t \log (u) \geq \Psi_\Upsilon(\log u) - \varphi (t)\; \text{ on } (0,\infty) \times \mathbb{Z},
\label{HarnackVoraussetzung}
\end{align}
where $\varphi: (0,\infty) \to [0,\infty)$ is continuous. Let $0 < t_1 < t_2$ and $x_1,x_2 \in \mathbb{Z}$. Then for every $N \in \mathbb{N}$ and every sequence of distinct points $(y_i)_{i = 0,1,\dots,N} \subset \mathbb{Z}$ satisfying $y_0 = x_1$, $y_N = x_2$ and $k(y_i-y_{i-1}) > 0, \  i \in \{1,\dots,N\},$ there holds
\begin{equation}
u(t_1,x_1) \leq u(t_2,x_2) \exp \Big(\int_{t_1}^{t_2} \varphi(t) \ \mathrm{d}t + \frac{2N}{t_2-t_1} \sum_{i=1}^N \frac{1}{k(y_i-y_{i-1})}\Big).
\label{Harnackungleichung}
\end{equation}
\end{satz}

Our proof follows the procedure in \cite[Theorem 6.1]{DKZ}.

\begin{proof}
Let $0 < t_1 < t_2$, $s \in [t_1,t_2] =: J$ and $x_1,x_2 \in \mathbb{Z}$. We first consider the case $N=1$. Then we have by assumption (\ref{HarnackVoraussetzung})
\begin{align*}
\log \frac{u(t_1,x_1)}{u(t_2,x_2)} &= \log \frac{u(t_1,x_1)}{u(s,x_1)} + \log \frac{u(s,x_1)}{u(s,x_2)} + \log \frac{u(s,x_2)}{u(t_2,x_2)} \\ &= -\int_{t_1}^s \partial_t \log u(t,x_1) \ \mathrm{d}t + \log \frac{u(s,x_1)}{u(s,x_2)} - \int_s^{t_2} \partial_t \log u(t,x_2) \ \mathrm{d}t \\ &\leq \int_{t_1}^{t_2} \varphi (t) \ \mathrm{d}t + \log \frac{u(s,x_1)}{u(s,x_2)} - \int_s^{t_2} \Psi_\Upsilon (\log u)(t,x_2) \ \mathrm{d}t \\ &\leq \int_{t_1}^{t_2} \varphi (t) \ \mathrm{d}t + \log \frac{u(s,x_1)}{u(s,x_2)} - \int_s^{t_2} k(x_1-x_2) \Upsilon\big(\log u(t,x_1)-\log u(t,x_2)\big) \ \mathrm{d}t \\ &= \int_{t_1}^{t_2} \varphi (t) \ \mathrm{d}t + \delta(s) - k(x_2-x_1) \int_s^{t_2} \Upsilon\big(\delta(t)\big) \ \mathrm{d}t,
\end{align*}
where we set 
\begin{align*}
\delta(t) = \log \frac{u(t,x_1)}{u(t,x_2)},\quad t \in J.
\end{align*}

We now choose $s \in J$ such that the continuous function $\omega$ defined by
\begin{align*}
\omega (t) = \delta (t) - k(x_2-x_1) \int_t^{t_2} \Upsilon\big( \delta(\tau)\big) \ \mathrm{d}\tau, t \in J,
\end{align*}
attains its minimum at $s$. In the proof of \cite[Theorem 6.1]{DKZ} it was shown that for this function the estimate
\begin{align*}
\omega(s) = \min_{t \in J} \omega (t) \leq \frac{2}{k(x_2-x_1) (t_2-t_1)}
\end{align*}
holds true. We thus have
\begin{align*}
\log \frac{u(t_1,x_1)}{u(t_2,x_2)} \leq \int_{t_1}^{t_2} \varphi (t) \ \mathrm{d}t + \frac{2}{k(x_2-x_1) (t_2-t_1)}, 
\end{align*}
which implies the statement for $N=1$. 

Turning now to the case $N>1$, we consider a sequence of points $(y_i)_{i = 0,1,\dots,N}$ such that $y_0 = x_1$, $y_N = x_2$ and $k(y_i-y_{i-1}) > 0, \ i \in \{1,\dots,N\}$. Defining the times $\tau_i = t_1 + i\frac{t_2-t_1}{N}, i=0,1,\dots,N,$ and employing the result for $N=1$ we obtain
\begin{align*}
\log \frac{u(t_1,x_1)}{u(t_2,x_2)} &= \sum_{i=1}^N \log \frac{u(\tau_{i-1},y_{i-1})}{u(\tau_i,y_i)} \\&\leq \sum_{i=1}^N \Big(\int_{\tau_{i-1}}^{\tau_i} \varphi (t) \ \mathrm{d}t + \frac{2}{k(y_i-y_{i-1})(\tau_i-\tau_{i-1})}\Big) \\&= \int_{t_1}^{t_2} \varphi (t) \ \mathrm{d}t + \frac{2N}{t_2-t_1} \sum_{i=1}^N \frac{1}{k(y_i-y_{i-1})}.
\end{align*}
This proves (\ref{Harnackungleichung}).
\end{proof}

\begin{korollar}
\label{KorollarHarnack}
Let $L$, $k$, $u$ and $\varphi$ be as in Theorem \ref{SatzHarnackUngl}.
Assume in addition that  $\varphi$ has a logarithmic singularity in $0$, i.e. $\varphi(t) \sim -c\log (t) \text{ as } t \to 0$ with some constant $c>0$. Then we can also allow for the choice $t_1=0$ in Theorem \ref{SatzHarnackUngl}.
\end{korollar}

\begin{bemerkungen}
\label{BemerkungHarnack}
(i) Consider the situation of Theorem \ref{SatzLiYau}. Then Theorem
\ref{SatzHarnackUngl} applies to the solution $u$, where $\varphi$ is the relaxation function associated with the $CD$-function $2F$. 
Corollary \ref{KorollarHarnack} is applicable if $\varphi$ possesses a logarithmic singularity in $0$. Our results from Section 3 show that this is the case for a large class of operators
including also $L_\beta$ ($\beta>0$) and the fractional discrete Laplacian $-(-\Delta)^\frac{\beta}{2}$ ($\beta\in (0,2)$) in one dimension.

(ii) The number $N \in \mathbb{Z}$ appearing in Theorem \ref{SatzHarnackUngl} can be understood as the number of steps that are made to get from $x_1$ to $x_2$. $\frac{t_2-t_1}{N}$ represents the step size in time which we chose to be equidistant.

(iii) Consider a function $u$ as in Theorem \ref{SatzHarnackUngl} and set $u_\varepsilon (t,x) = u(t,x)+\varepsilon, \varepsilon>0$. Then condition (\ref{HarnackVoraussetzung}) on $u$ can be replaced by the same condition for the family $u_\varepsilon$, $\varepsilon\in (0,\varepsilon_0]$. Indeed, consider the last calculation in the proof of Theorem \ref{SatzHarnackUngl}. Replacing $\log u$ by $\log u_\varepsilon$ gives
\begin{align*}
\log\big(u(t_1,x_1)+\varepsilon\big) \leq \log\big(u(t_2,x_2)+\varepsilon\big) + \int_{t_1}^{t_2} \varphi(t) \ \mathrm{d}t + \frac{2N}{t_2-t_1} \sum_{i=1}^N \frac{1}{k(y_i-y_{i-1})}, \ \varepsilon\in (0,\varepsilon_0].
\end{align*}
Sending $\varepsilon \to 0$ and applying the exponential function yields (\ref{Harnackungleichung}).
\end{bemerkungen}

A natural question regarding the Harnack inequality (\ref{Harnackungleichung}) is which choice of intermediate points $(y_i)_{i=0,\dots,N}$ gives the best upper bound. In other words, aiming to find the optimal upper bound in this inequality is equivalent to minimizing the occurring sum
\begin{equation}
\mathcal{S} = \mathcal{S} \big((y_i)_{i=0}^N \big) := N \sum_{i=1}^N \frac{1}{k(y_i-y_{i-1})},
\label{HarnackSumme}
\end{equation}
which means that we look for a sequence of points $(y_i)_{i=0,1,\dots,N}$ with $N \in \N$ that minimizes (\ref{HarnackSumme}). 

In the case of a general kernel without any monotonicity assumption this minimization problem turns out to be quite difficult as it could be of advantage to jump away from $x_2$. We illustrate this in the following example.

\begin{beispiel}
For $0<\beta_1<1<\beta_2<2$ consider the kernel $k:\mathbb{N} \to [0,\infty)$ given by
\begin{align*}
k(j) = \begin{cases}
\frac{1}{j^{1+\beta_1}}, j \text{ even,}\\
\frac{1}{j^{1+\beta_2}}, j \text{ odd}.
\end{cases}
\end{align*}
We further set $k(0)=0$ and extend $k$ to a symmetric kernel on $\mathbb{Z}$. Let $x_1,x_2 \in \mathbb{Z}$ with $x_1<x_2$ such that $M=x_2-x_1$ is odd. Choosing $N=1$  in (\ref{HarnackSumme}) (which means that we jump directly from $x_1$ to $x_2$) gives $\mathcal{S} = M^{1+\beta_2}.$
For $N=M$ (which means that we make $M$ steps of size $1$) we get a better estimate as we then have $\mathcal{S} = M^2$.
Both estimates can be improved by first making one single good step and then jumping directly to $x_2$. This good step can also be in the wrong direction as we will see. Therefore, set $N=2$ and define $y_1 = x_1-1$. Then $\mathcal{S} = 2 \big(1+(M+1)^{1+\beta_1}\big).$
Choosing $M$ big enough, we see that this estimate further improves our already found upper bounds as we have $\beta_1<\beta_2$.
\end{beispiel}

We now want to focus on kernels that are non-increasing on $\mathbb{N}$ as these are of main interest in this article. For these kernels one can show that jumping away from $x_2$ will not give a better estimate. Indeed, consider w.l.o.g. for $N=2$ the three points $y_{0},y_1,y_2 \in \mathbb{Z}$ chosen such that $y_1<y_0<y_2$. Then 
\begin{align*}
\mathcal{S} \big((y_i)_{i=0}^2 \big) = 2\Big(\frac{1}{k(y_0-y_1)} + \frac{1}{k(y_2-y_1)}\Big) > \frac{1}{k(y_2-y_0)}.
\end{align*}
This means that jumping directly from $y_0$ to $y_2$ would give a better upper bound.

\begin{proposition}\label{prop_Harnacksum}
Let the kernel $k$ be non-increasing on $\N$ and set $q_0=\min\{n \in \mathbb{N}: k(n)=0\}$ if $k$ has finite support and $q_0=\infty$ otherwise.
Let $x_1,x_2 \in \Z$ with $x_1<x_2$ and set $M=x_2-x_1$. Assume that there exists a convex $C^1$-mapping $q:[1,\min \{M,q_0\})\to(0,\infty)$ such that $q(n)=\frac{1}{k(n)}$ for any $n \in \N \cap [1,\min \{M,q_0-1\}]$. Then the following two assertions hold true.
\begin{itemize}
\item[(i)] If $k(M)>0$ and $\frac{q'(y)y}{q(y)}\leq 2$ holds for all $y \in [1,M]$, then \eqref{HarnackSumme} is minimized by the choice of intermediate points $y_0=x_1$ and $y_1=x_2$.
\item[(ii)] If $\frac{q'(y)y}{q(y)}\geq 2$ holds for all $y \in [1,\min\{M, q_0\})$, then \eqref{HarnackSumme} is minimized by the choice of intermediate points $y_i=x_1+i$ for $i \in \{0,...,M\}$.
\end{itemize}
\end{proposition}
\begin{proof}
In order to minimize the expression \eqref{HarnackSumme} we can assume w.l.o.g.\ by monotonicity of the kernel that the set $\big(y_i\big)_{i=0}^N$ is ordered in the sense that
$
x_1 =y_0 <y_{i-1}<y_i< y_N=x_2$
for every $i \in \{2,...,N-1\}$. Consequently,  
$
\sum\limits_{i=1}^N |y_{i}-y_{i-1}| = M.
$
By the properties of $q$, we have by Jensen's inequality
\begin{align*}
\mathcal{S}\big((y_i)_{i=0}^N\big) = N \sum\limits_{i=1}^N q\big( y_i-y_{i-1} \big) \geq N^2 q\Big( \frac{1}{N}\sum\limits_{i=1}^N (y_i - y_{i-1}) \Big) = N^2 q\big(\frac{M}{N}\big).
\end{align*}
We consider the mapping $h(r)=r^2 q\big(\frac{M}{r}\big)$, $r \in \big(\max \big\{1,\frac{M}{q_0}\big\},M \big]$ and observe that
\begin{align*}
h'(r)= 2r q\big(\frac{M}{r}\big) - q'\big(\frac{M}{r}\big)M.
\end{align*}
Consequently, $h$ is non-decreasing (non-increasing) if and only if
\begin{equation*}
2 \geq \, (\le )\,\frac{q'\big(\frac{M}{r}\big)M}{q\big(\frac{M}{r}\big)r},
\end{equation*}
which shows both assertions, as $\mathcal{S}\big((y_i)_{i=0}^N\big)= N^2 q\big(\frac{M}{N}\big)$ for the two specific choices of intermediate points considered in (i) and (ii), respectively.
\end{proof}
\begin{bemerkung}
The condition on $q$ in the above proposition already appeared in the context of $CD$-functions in \cite{WEB}. Observe that the condition $\frac{q'(y)y}{q(y)}\geq (\le)\,2$ is equivalent to saying that $y\mapsto \frac{q(y)}{y^2}$ is non-decreasing (non-increasing), see also \cite[Remark 3.4(i)]{WEB}. In particular, the condition of Proposition \ref{prop_Harnacksum}(i) means that $n \mapsto n^2k(n)$, $n \in \mathbb{N}$, is non-decreasing, which clearly implies that $k$ has no finite second moment. 
As a consequence, we note that for a kernel with finite second moment the condition of 
Proposition \ref{prop_Harnacksum}(i) cannot be satisfied for every $M>0$.  
\end{bemerkung}
\begin{beispiel}\label{HarnackExamples}
(i) In case of the power-type kernel $k_\beta$, the mapping $q$ from Proposition \ref{prop_Harnacksum} is given by $q(y)=y^{1+\beta}$, which is convex for any $\beta \in (0,2)$. Clearly, we have that 
\begin{align*}
\frac{q'(y)y}{q(y)}= 1+\beta,
\end{align*}
which is greater or equal than $2$ if $\beta\geq 1$ and less than or equal than $2$ if $\beta\leq 1$. Therefore one chooses always the direct jump  if $\beta \in (0,1)$ and takes the single steps if $\beta>1$ in order to minimize (\ref{HarnackSumme}). In the case of $\beta=1$ both the single steps and the direct jump lead to the optimal value in \eqref{HarnackSumme}.

(ii) Despite the similarity of the kernel $k_\beta^*$ associated with the fractional discrete Laplacian in one dimension, the situation is slightly different from the power-type kernel in (i). By the representation formula \eqref{FrakLaplacekernel}, we can write
\begin{equation*}
q(y)=c_\beta \frac{\Gamma\big(y+1+\frac{\beta}{2}\big)}{\Gamma\big(y-\frac{\beta}{2}\big)}, \quad y \geq 1,
\end{equation*}
where $\Gamma$ denotes the Gamma function and $c_\beta>0$ is a constant that comes from \eqref{FrakLaplacekernel}. We have
\begin{align*}
q'(y)& = c_\beta \Big( \frac{\Gamma'(y+1+\frac{\beta}{2})}{\Gamma(y-\frac{\beta}{2})} - \frac{\Gamma'(y-\frac{\beta}{2})}{\Gamma(y-\frac{\beta}{2})}\frac{\Gamma(y+1+\frac{\beta}{2})}{\Gamma(y-\frac{\beta}{2})} \Big)\\
&  = q(y) \Big( \psi\big(y+1+\frac{\beta}{2}\big) - \psi\big(y-\frac{\beta}{2}\big)\Big),
\end{align*}
where $\psi$ denotes the Digamma function defined by $\psi(x)=\frac{\Gamma'(x)}{\Gamma(x)}$. Since the Gamma function is logarithmically convex, $ \psi$ is 
non-decreasing and so is $q$. As a byproduct, monotonicity of the kernel $k_\beta^*$ follows (cf. Remark \ref{BemerkungFrakLaplace}). As to the second derivative, we observe
\begin{equation*}
q''(y) = q(y) \Big( \psi'\big(y+1+\frac{\beta}{2}\big) - \psi'\big(y-\frac{\beta}{2}\big) + \Big(\psi\big(y+1+\frac{\beta}{2}\big) - \psi\big(y-\frac{\beta}{2}\big)\Big)^2\Big).
\end{equation*}
We deduce from Theorem 1 in \cite{QiLuo} (choosing $\lambda=\frac{1}{s-t}$ in the notation of \cite{QiLuo}) that $q$ is convex. Assuming now that $\beta\geq 1$, we infer from the monotonicity of $\psi$ that
\begin{align*}
\frac{q'(y)y}{q(y)}&  = y \Big( \psi\big(y+1+\frac{\beta}{2}\big) - \psi\big(y-\frac{\beta}{2}\big)\Big) \geq y \Big( \psi\big(y+\frac{3}{2}\big) - \psi\big(y-\frac{1}{2}\big)\Big)\\
&= y\Big(\frac{1}{y+\frac{1}{2}}+\frac{1}{y-\frac{1}{2}}\Big) = \frac{2y^2}{y^2-\frac{1}{4}}> 2,
\end{align*}
where we applied the recurrence relation $\psi(x+1)=\psi(x)+\frac{1}{x}$ twice.  Consequently, assertion (ii) of Proposition \ref{prop_Harnacksum} applies. In particular, in contrast to the case of power-type kernels, we now have a strict inequality for $\beta=1$ in the corresponding assumption of Proposition \ref{prop_Harnacksum}(ii).  
Since $ \psi\big(y+1+\frac{\beta}{2}\big) - \psi\big(y-\frac{\beta}{2}\big)\to \frac{1}{y}$
as $\beta\to 0+$ for all $y\ge 1$ and by continuity,
for every $M\in \N$ with $M \geq 2$ 
we find a minimal $\beta_M\in (0,1)$ such that Proposition \ref{prop_Harnacksum}(ii) applies
for $\beta \in [\beta_M,1)$, that is, in order to optimize \eqref{HarnackSumme}, one takes single steps. This is in contrast to the situation of the  power-type kernel as described above. Note that $\beta_M\to 1$ as $M \to \infty$. In fact, $\beta_M$ is non-decreasing
as a function of $M$,
by definition, and $\beta_M<1$. Thus the limit $\beta'=\lim_{M\to \infty}\beta_M\le 1$ exists. Since for every $M\in \iN$, $M\ge 2$, the single steps of step size $1$ minimize the sum \eqref{HarnackSumme}, and $k_\beta^*$ is comparable to $k_\beta$, we have
$M^2\le q(M)\le c(\beta')M^{1+\beta'}$ for all $M\ge 2$. Sending $M\to \infty$ shows that
we must have $\beta'=1$.

(iii) Taking the exponential kernel $k(j)=e^{-\alpha |j|}$, $\alpha>0$, we have that $q(y)=e^{\alpha y}$ (which is convex for $\alpha>0$) and hence
\begin{align*}
\frac{q'(y)y}{q(y)} = \alpha y.
\end{align*}
If $\alpha\geq 2$, assertion (ii) of Proposition \ref{prop_Harnacksum} is satisfied due to $y\geq 1$. If instead $\alpha<2$, the best possible choice in \eqref{HarnackSumme} depends on the distance between $x_1$ and $x_2$. If, for instance, $x_1,x_2\in \Z$ are such that $\alpha |x_1-x_2|\leq 2$, then assertion (i) in Proposition \ref{prop_Harnacksum} holds true. However, there are also situations where neither the single steps nor the direct jump lead to the optimal choice in \eqref{HarnackSumme}. Indeed, let $\alpha=1$ and $M:=|x_1-x_2|$ be even. Then one readily checks that the mapping $h$ from the proof of Proposition \ref{prop_Harnacksum} achieves its minimum at $y= \frac{M}{2}$. Then \eqref{HarnackSumme} can be estimated as follows
\begin{align*}
\mathcal{S} \geq N^2 q\big(\frac{M}{N}\big) \geq  \frac{M^2}{4}e^2,
\end{align*}
which corresponds to $\frac{M}{2}$ equidistantly chosen intermediate points  $\big(y_i\big)_{i=1}^\frac{M}{2}$.
\end{beispiel}

As the case of the power-type kernel $k_\beta$ ($\beta\in (0,2)$) is of particular importance for us, we state the Harnack inequality from Theorem \ref{SatzHarnackUngl} in this special case,
taking into account the findings from Example \ref{HarnackExamples}(i).

\begin{satz}
Let $\beta \in (0,2)$, $0 \leq t_1 < t_2$ and $x_1,x_2 \in \mathbb{Z}$. Suppose that $u: [0,\infty) \times \mathbb{Z} \to (0,\infty)$  is a bounded (and $C^1$ in time) solution of the equation $\partial_t u (t,x) - L_\beta u(t,x) = 0$ on $(0,\infty) \times \mathbb{Z}$. Then there holds
\begin{equation*}
u(t_1,x_1) \leq u(t_2,x_2) \exp \Big(\int_{t_1}^{t_2} \varphi(t) \ \mathrm{d}t + \frac{2 \vert x_1-x_2 \vert^{\min\{1+\beta,2\}}}{t_2-t_1}\Big),
\end{equation*}
where $\varphi$ is a relaxation function of the type described in Lemma \ref{LemmaRelaxFunc}, with corresponding parameters given by Theorem \ref{SatzCDFunctionPotenzkern}.
\end{satz}

Finally, we want to use Theorem \ref{SatzHarnackUngl} to derive heat kernel estimates. The proof is inspired by \cite[Theorem 7.6]{BHL}, where the authors proved heat kernel bounds on locally finite graphs. Apart from the setting, we emphasize that we can formulate evidently stronger results as we are also able to treat the case $t_1=0$ in Theorem \ref{SatzHarnackUngl}.

Let $\alpha >0$ and $k$ be a kernel with finite $\alpha$-th moment. Let $L$ be the operator generated by $k$. The heat kernel $p_t (x,y)$ for the equation $\partial_t u -L u = 0$ on $(0,\infty) \times \mathbb{Z}$ can be defined as solution of the equation
\begin{align*}
\begin{cases}
p_0(x,y) = \delta (x,y), \ x,y \in \Z, \\
\partial_t p_t (x,y) = \sum_{z \in \mathbb{Z}} k(x-z) \big(p_t(z,y)-p_t(x,y)\big), \ x,y \in \Z,
\end{cases}
\end{align*}
where $\delta (x,y) = 1$ if $x=y$ and $\delta (x,y)=0$ else, cf.\ e.g.\ \cite{DEL}. As stated in \cite{DEL}, one can show that for any $x,y \in \mathbb{Z}$ there holds $p_t(x,y) = p_t (y,x)$.
With this we can prove the following theorem.

\begin{satz}\label{thm:heatkernelbounds}
Let $\alpha >0$ and $k$ be a kernel with finite $\alpha$-th moment. Let $L$ be the operator generated by $k$. Suppose that $L$ satisfies $CD_\Upsilon (0,F)$ with a $CD$-function $F$ that satisfies $F(x) \sim c_1 x^\gamma$ as $x \to 0$ and $F(x) \sim c_2 e^{\delta x}$ as $x \to \infty$, where $\gamma \geq 2$ and $\delta, c_1,c_2 >0$. Let $p_t(x,y), \ x,y \in \mathbb{Z},$ be the heat kernel of the equation $\partial_t u - Lu = 0$ on $(0,\infty) \times \mathbb{Z}$. Then there exists a constant $c=c(F)>0$ such that 
for any $\nu>0$ there holds the two-sided heat kernel estimate
\begin{align}\label{eq:heatkernelbounds}
\exp \Big(- \Big(\Lambda(t)+\frac{2\vert x-y \vert^2}{k(1) t}\Big)\Big) \leq p_t (x,y) \leq \frac{\exp\Big( C_0 \nu + \frac{2}{|k|_1}\Big)}{2 \lfloor \sqrt{c(t,\nu)-t)} \rfloor}, \ x,y \in \mathbb{Z}, t>0,
\end{align}
where $\Lambda:(0,\infty)\to (0,\infty)$ is given by
\begin{equation*}
\Lambda(t)= \begin{cases}
\frac{t}{\delta}\big(1- \log(2^\gamma c \delta^{1-\gamma}t)\big) &, t\leq t_*\\
\frac{\log(t)}{c e^2} + \tau_1 &, t> t_* \text{ and } \gamma=2 \\
\ \frac{((\gamma-1) t)^\frac{\gamma-2}{\gamma-1}}{(c e^2)^\frac{1}{\gamma-1} (\gamma-2)} + \tau_2 &,  t> t_* \text{ and } \gamma>2,
\end{cases}
\end{equation*}
with constants $\tau_i=\tau_i(c,\gamma,\delta)\in \iR$ for $i \in \{1,2\}$, 
$t_*=\frac{\delta^{\gamma-1}}{2^\gamma c e^2}$,
 $C_0 =C_0(c,\gamma,\delta)>0$, and where $c(t,\nu) =\big(\frac{\gamma-2}{\gamma-1}\nu + t^\frac{\gamma-2}{\gamma-1}\big)^\frac{\gamma-1}{\gamma-2}$ if $\gamma>2$ and  $c(t,\nu)=e^\nu t$ if $\gamma =2$.
\end{satz}

\begin{proof}
First, recall that the asymptotic assumptions on the $CD$-function $F$ ensure that there exists a constant $c=c(F)>0$ such that $L$ also satisfies $CD_\Upsilon(0,\tilde{F})$, with $\tilde{F}=\frac{c}{2}\hat{F}$ and $\hat{F}$ given by \eqref{eq:Fhat}, where we choose $M=\frac{2}{\delta}$. Let $\varphi$ be the relaxation function corresponding to $2 \tilde{F}$ given by Lemma \ref{LemmaRelaxFunc}. 
 
We first prove the lower bound in \eqref{eq:heatkernelbounds}. For that purpose, we  apply Theorem \ref{SatzLiYau} to the heat kernel in order to make use of Theorem \ref{SatzHarnackUngl}. For $x,y \in \mathbb{Z}, \ x>y,$ and $N = x-y$ we choose in \eqref{HarnackSumme} $y_0 = y$, $y_N = x$, $t_1 =0$ and $t_2 = t$. Note, that we can choose $t_1 = 0$ since $\varphi$ has a logarithmic singularity in $0$, by Lemma \ref{LemmaRelaxFunc}. We further set $y_i = y_{i-1}+1 \ (i=1,\dots,N-1)$. This yields
\begin{equation*}
1=p_0(y,y) \leq p_t(x,y) \exp \Big(\int_0^t \varphi (s) \ \mathrm{d}s \Big) \exp \Big( \frac{2 \vert x-y \vert^2}{k(1) t} \Big).
\end{equation*}
If $t\leq t_*$, we have by substitution
\begin{align*}
\int_0^t \varphi (s) \mathrm{d}s = -\frac{1}{c \delta^2 M^\gamma}
\int_0^{c \delta M^\gamma t} \log (s) \, \mathrm{d}s = 
\frac{t}{\delta} \big(1-\log (2^\gamma c \delta^{1-\gamma} t) \big).
\end{align*}
If instead $t>t_*$,  Lemma \ref{LemmaRelaxFunc}, together with the previous calculation, gives
\begin{align*}
\int_0^t \varphi (s) \mathrm{d}s \leq 
\frac{t_*}{\delta} \big(1-\log (2^\gamma c \delta^{1-\gamma} t_*) \big)+
\frac{1}{\big(ce^2(\gamma-1)\big)^\frac{1}{\gamma-1}}\int_{t_*}^t s^{-\frac{1}{\gamma-1}}\mathrm{d}s.
\end{align*}
Thus, the remaining cases in the definition of $\Lambda$ follow from a straightforward calculation.

For the proof of the upper bound we first note that one readily verifies from the explicit representation of $\varphi$ that there exists some constant $C_0>0$ such that $\varphi(s)\leq C_0 \,s^{-\frac{1}{\gamma-1}}$ for any $s>0$. Besides that, we  again employ Theorem \ref{SatzHarnackUngl}. 
We consider  $t_1 =t$ and $t_2 = c(t,\nu)$. For $x \in \mathbb{Z}$, we set $x_1=x$ and $x_2 = z$, where we choose $z \in \mathbb{Z}$ such that $\vert x-z \vert \leq \sqrt{c(t,\nu)-t}$. Analogously to the proof of the lower estimate, this gives
\begin{align*}
p_t(x,y) &\leq p_{c(t,\nu)} (z,y) \exp \Big(\int_t^{c(t,\nu)} \varphi (s) \ \mathrm{d}s \Big) \exp \Big(\frac{2 \vert x-z \vert ^2}{k(1) (c(t,\nu)-t)} \Big) \\ &\leq p_{c(t,\nu)} (z,y) \exp \Big(C_0 \int_t^{c(t,\nu)} s^{-\frac{1}{\gamma-1}} \ \mathrm{d}s \Big) \exp \Big(\frac{2}{k(1)}\Big)\\
&= \exp\Big( C_0 \nu + \frac{2}{k(1)}\Big) p_{c(t,\nu)} (z,y).
\end{align*}
Here we used the identity
\begin{equation*}
\int_t^{c(t,\nu) } s^{-\frac{1}{\gamma-1}} \ \mathrm{d}s = \nu,
\end{equation*}
which holds by the choice of $c(t,\nu)$. This yields
\begin{align*}
p_t (x,y) \leq \frac{\exp\Big( C_0 \nu + \frac{2}{k(1)}\Big)}{2 \lfloor \sqrt{c(t,\nu)-t} \rfloor} \sum_{ \{z \in \mathbb{Z}: \vert x-z \vert \leq \sqrt{c(t,\nu)-t} \}} p_{c(t,\nu)} (z,y) \leq \frac{\exp\Big( C_0 \nu + \frac{2}{k(1)}\Big)}{2 \lfloor \sqrt{c(t, \nu)-t} \rfloor}. 
\end{align*}
\end{proof}
\begin{bemerkungen}
(i) Note that for $\gamma>2$ we can write 
\begin{equation*}\label{eq:cminust}
c(t,\nu)-t = h\Big(\frac{\gamma-2}{\gamma-1}\nu +t^\frac{\gamma-2}{\gamma-1}\Big)- h\Big(t^\frac{\gamma-2}{\gamma-1}\Big),
\end{equation*}
where $h(r)=r^\frac{\gamma-1}{\gamma-2}$, $r>0$. In order to study the long-time behavior, the dominating term in the respective Taylor expansion at $t^\frac{\gamma-2}{\gamma-1}$ is given by
\begin{equation*}
\frac{\gamma-2}{\gamma-1}\nu \, h'(t^\frac{\gamma-2}{\gamma-1})= \nu \,t^\frac{1}{\gamma-1}
\end{equation*}
and hence $c(t,\nu)-t$ behaves for large $t$ as $t^\frac{1}{\gamma-1}$.

(ii) For simplicity, we have presented Theorem \ref{thm:heatkernelbounds} in such a way that the sum \eqref{HarnackSumme} is always estimated by taking single steps. However, the optimal choice in \eqref{HarnackSumme} highly depends on the kernel as we have seen in Proposition \ref{prop_Harnacksum} and Example \ref{HarnackExamples}. In particular, Example \ref{HarnackExamples} (i) shows that for $\beta<1$ the direct jump is the best choice. In this case, we can improve \eqref{eq:heatkernelbounds} by
\begin{equation}\label{eq:improvedheatbounds}
\exp \Big(-\Big(\Lambda(t)+\frac{2\vert x-y \vert^{1+\beta}}{ t}\Big)\Big) \leq p_t (x,y) \leq \frac{C}{2 \lfloor (c(t,\nu)-t))^\frac{1}{1+\beta} \rfloor}, 
\end{equation}
where $x,y \in \mathbb{Z}$ and $t>0$. Note that for $\beta<1$ we have $\gamma=\frac{1+2\beta}{\beta}$ by Theorem \ref{SatzCDFunctionPotenzkern}. With the findings from the first part of this remark we conclude that $c(t,\nu)-t$ behaves like $t^\frac{\beta}{1+\beta}$ for large $t$. The asymptotic behaviour of the right-hand side in equation \eqref{eq:improvedheatbounds} for large $t$ is then like $t^\frac{\beta}{(1+\beta)^2}$.

(iii) For the long-time behaviour of the lower bound in \eqref{eq:heatkernelbounds}, the definition of $\Lambda$ yields that  the asymptotic behaviour as $t \to \infty $ in the left-hand side of \eqref{eq:heatkernelbounds} is like $t^{-\frac{1}{ce^2}}$ if $\gamma=2$ and like $\exp\big(- \tilde{c}t^\frac{\gamma-2}{\gamma-1}\big)$ with some 
constant $\tilde{c}>0$ if $\gamma>2$.
\end{bemerkungen}


$\mbox{}$
{\footnotesize

$\mbox{}$


}

\end{document}